\documentclass[12pt]{article}
\usepackage{lgrind}

\usepackage{makeidx}

\makeatletter
\def\printnotation{{%
\def\indexname{Index of notation}
\begin{theindex}
\@input{\jobname.ntn}
\end{theindex}
}}
\makeatother

\makeglossary

\usepackage{amsfonts,amssymb,stmaryrd,amscd,amsmath,latexsym,amsbsy}
\usepackage{amsfonts,amssymb,graphicx,amsthm}
\usepackage[all,cmtip,curve, knot,frame]{xy} 
\usepackage{amssymb}
\usepackage{amsfonts}
\usepackage{latexsym}
\usepackage{epstopdf}
\usepackage{tikz}
\usetikzlibrary{matrix,arrows}

\newtheorem{theorem}{Theorem}[section]
\newtheorem{lemma}[theorem]{Lemma}

\newtheorem{proposition}[theorem]{Proposition}
\newtheorem{corollary}[theorem]{Corollary}
\theoremstyle{definition}
\newtheorem{definition-proposition}[theorem]{Definition-Proposition}
\newtheorem{definition}[theorem]{Definition}
\newtheorem{example}[theorem]{Example}
\newtheorem{remark}[theorem]{Remark}

\newcommand{\rightloop}{%
           \mathrel{\raisebox{.1em}{%
           \reflectbox{\rotatebox[origin=c]{-90}{$\circlearrowright$}}}}}
\newcommand{\leftloop}{%
           \mathrel{\raisebox{.1em}{%
           \reflectbox{\rotatebox[origin=c]{90}{$\circlearrowright$}}}}}
\newcommand{\bigboxtimes}[1]{\underset{#1}{\boxtimes}}

\newcommand{\dd}{{\bf d}}
\newcommand{\edit}[1]{}
\newcommand{\Vect}{\operatorname{Vect}}

\newcommand{\Hom}{\operatorname{Hom}}

\newcommand{\Rep}{\text{Rep}}

\newcommand{\tr}{\text{tr}\,}
\newcommand{\ad}{\text{ad}\,}

\newcommand{\cC}{\mathcal{C}}
\newcommand{\cD}{\mathcal{D}}
\newcommand{\cO}{\mathcal{O}}

\newcommand{\dS}{\Big/\hspace{-5pt}\Big/}

\newcommand{\g}{\mathfrak{g}}

\newcommand{\ot}{\otimes}
\newcommand{\bt}{\boxtimes}
\newcommand{\id}{\operatorname{id}}

\newcommand{\CC}{\mathbb{C}}
\newcommand{\NN}{\mathbb{N}}
\newcommand{\ZZ}{\mathbb{Z}}

\newcommand{\coev}{\operatorname{coev}}
\newcommand{\ev}{\operatorname{ev}}

\newcommand{\Mat}{\operatorname{Mat}}

\newcommand{\tq}{q}

\def\HH{\hbox{${\mathcal H}$\kern-5.2pt${\mathcal H}$}}

\begin{document}
\title{Quantized Multiplicative Quiver Varieties}
\author{David Jordan}
\maketitle
\begin{abstract}
Beginning with the data of a quiver $Q$, and its dimension vector $\dd$, we construct an algebra $\cD_q=\cD_q(\Mat_\dd(Q))$, which is a flat $q$-deformation of the algebra of differential operators on the affine space $\Mat_\dd(Q)$. The algebra $\cD_q$ is equivariant for an action by a product of quantum general linear groups, acting by conjugation at each vertex. We construct a quantum moment map for this action, and subsequently define the Hamiltonian reduction $A^\lambda_\dd(Q)$ of $\cD_q$ with moment parameter $\lambda$. We show that $A^\lambda_\dd(Q)$ is a flat formal deformation of Lusztig's quiver varieties, and their multiplicative counterparts, for all dimension vectors satisfying a flatness condition of Crawley-Boevey: indeed the product on $A^\lambda_\dd(Q)$ yields a Fedosov quantization the of symplectic structure on multiplicative quiver varieties. As an application, we give a description of the category of representations of the spherical double affine Hecke algebra of type $A_{n-1}$, and its generalization constructed by Etingof, Oblomkov, and Rains, in terms of a quotient of the category of equivariant $\cD_q$-modules by a Serre sub-category of aspherical modules.
\end{abstract}


\section{Introduction}
Let us consider a reductive algebraic group $G$, with Lie algebra $\mathfrak{g}$.  Central to the representation theory of $G$ is the following organizing diamond:

\begin{equation}\label{Uqdiag}\xymatrix{ & U_q(\mathfrak{g}) \ar@{~>}[dl]_{\underset{\textrm{limit $q\to 1$}}{\textrm{\scriptsize quasi-classical }}} \ar@{~>}[dr]^{\underset{\textrm{limit $q\to 1$}}{\textrm{\scriptsize classical}}} & \\
                     U(\mathfrak{g}) \ar@{~>}[dr]_{\underset{\textrm{graded}}{\textrm{\scriptsize associated }}} && O(G) \ar@{~>}[dl]^{\underset{\textrm{degeneration}}{\textrm{\scriptsize rational }}}\\\
                     &  S(\mathfrak{g}) & }
                     \end{equation}

The quantum group $U_q(\mathfrak{g})$ (more precisely, its locally finite subalgebra $U'_q(\mathfrak{g})$) depends on a quantum parameter $q$, as well as a Rees parameter $t$; degenerating $q$ and $t$ appropriately, we can recover either the universal enveloping algebra $U(\mathfrak{g})$, or the coordinate algebra $O(G)$ of the group $G$.  The main idea of this thesis is to apply the same organizing diamond to the algebro-geometric constructions surrounding representations of quivers, with the end-goal of producing a new quantization of the coordinate algebra of a multiplicative quiver variety.

Let $Q=(V,E)$ denote a connected quiver, with vertex set $V$, and directed edge set $E$.  For $e\in E$, let $\alpha=\alpha(e)$ and $\beta=\beta(e)$ denote the tail and head of $e$, respectively.  Fix a function $\dd:V\to\ZZ_{\geq 0}, v\mapsto d_v$ (called a dimension vector), and consider the following affine variety and affine algebraic group, respectively:
$$\Mat_{\dd}(Q):=\prod_{e\in E} \Mat(\CC^{d_\alpha},\CC^{d_\beta}), \,\,\, \mathbb{G}^\dd:=\prod_{v\in V} GL(\CC^{d_v}).$$
We let $\mathbb{G}^\dd$ act on $X\in\Mat_\dd(Q)$ by change of basis at each vertex,
$$(g.X)_e:= g_{\beta(e)} X_e g_{\alpha(e)}^{-1}.$$
The doubled quiver $\overline{Q}=(V,\overline{E}=E\cup E^\vee)$ is built from $Q$ by adding an adjoint arrow \mbox{$\beta(e)\xrightarrow{e^\vee} \alpha(e)\in E^\vee$}, for each $e\in E$.  We have canonical isomorphisms,
$$ T^*\Mat(\CC^{d_\alpha},\CC^{d_\beta})\cong \Mat(\CC^{d_\alpha},\CC^{d_\beta}) \times \Mat(\CC^{d_\beta},\CC^{d_\alpha}),$$ with the standard symplectic pairing given by:
$$(X,Y)=\tr (X_eY_{e^\vee}-Y_eX_{e^\vee}).$$
Taken together, these give an identification $T^*\Mat_\dd (Q)\cong\Mat_\dd(\overline{Q})$.

The construction of multiplicative quiver varieties involves a certain open subset $T^*\Mat_\dd(Q)^\circ$ defined by a non-degeneracy condition, while the construction of quantized quiver varieties involves the algebra $\cD(\Mat_\dd(Q))$ of polynomial differential operators on $\Mat_\dd(Q)$, a quantization of the the symplectic structure on $T^*\Mat_\dd(Q)$.  We begin by seeking an algebra $\cD_q(\Mat_\dd(Q))$, situated as follows:
\begin{equation}\label{Dqdiag}\xymatrix@!C@C=-15pt@R=40pt{ & \mathcal{D}_q(\Mat_\dd(Q)) \ar@{~>}[dl]_{\underset{\textrm{limit $q\to 1$}}{\textrm{\scriptsize quasi-classical }}} \ar@{~>}[dr]^{\underset{\textrm{limit $q\to 1$}}{\textrm{\scriptsize classical}}} & \\
                     \cD(\Mat_\dd(Q)) \ar@{~>}[dr]_{\underset{\textrm{graded}}{\textrm{\scriptsize associated }}} && \cO(T^*\Mat_\dd(Q)^\circ) \ar@{~>}[dl]^{\underset{\textrm{degeneration}}{\textrm{\scriptsize rational }}}\\\
                     &  \cO(T^*\Mat_\dd(Q)) & }
                     \end{equation}
That is, the algebra $\cD_q=\mathcal{D}_q(\Mat_\dd(Q))$ that we construct depends on a parameter $q$, as well as a Rees parameter $t$; by degenerating $q$ and $t$ appropriately, we can recover either the algebra $D(\Mat_\dd(Q))$ or $O(T^*\Mat_\dd(Q)^\circ)$.  We give a straightforward presentation of $\cD_q$ by generators and relations, and show that it admits a PBW basis of ordered monomials analogous to the usual basis of $\cD(\Mat_\dd(Q))$.

There is more than simply an analogy between the diagrams \eqref{Uqdiag} and \eqref{Dqdiag}: indeed each vertex of \eqref{Dqdiag} receives a ``moment map" from the corresponding vertex of \eqref{Uqdiag}.  We briefly recall each of these maps here, directing the reader to Section \ref{prelims} for further details.

Recall that the action of $\mathbb{G}^\dd$ on $T^*\Mat_\dd(Q)$ is Hamiltonian; on the level of coordinate algebras, this means we have a $\mathbb{G}^\dd$-equivariant homomorphism, $$\mu^\#:S(\g)\to \cO(T^*\Mat_\dd(Q)),$$ of Poisson algebras, inducing the $\mathbb{G}^\dd$-action, via the Poisson bracket.  Likewise, there is a quantized moment map, a $\mathbb{G}^\dd$-equivariant homomorphism, $$\hat\mu^\#:U(\mathfrak{g})\to\cD(\Mat_\dd(Q)),$$  of algebras, inducing the $\mathbb{G}^\dd$-action by via the Lie bracket.  Finally, the ``group-valued" moment map for the action of $\mathbb{G}^\dd$ on $T^*\Mat_\dd(Q)^\circ$ is, on the level of coordinate algebras, a $\mathbb{G}^\dd$-equivariant homomorphism, $$\widetilde{\mu}^\#:\cO(\mathbb{G}^\dd)\to\cO(T^*\Mat_\dd(Q)^\circ)$$ of quasi-Poisson algebras, which again induces the action of $\mathbb{G}^\dd$ on $T^*\Mat_\dd(Q)^\circ$.

In Section \ref{q-moment}, we construct an equivariant homomorphism, $$\mu_q^\#: U'_q(\g^\dd)\to\cD_q(\Mat_\dd(Q)),$$ which is a quantum moment map in the sense of \cite{L}, \cite{VV}: the action of $U_q(\g^\dd)$ on $\cD_q(\Mat_\dd(Q))$ is given by a variant of the adjoint action.  Each of these moment maps is compatible with the simultaneous degenerations of source and target; we can express these relationships in the following diagram (we abbreviate $X=\Mat_\dd(Q)$):
\begin{equation}\label{muqdiag}
\xymatrix@!C@=10pt{ & U_q(\mathfrak{g}^\dd) \ar@{~>}[dl]\ar@{~>}[dddd] \ar@{->}[rr]^{\mu^\#_q}&&\mathcal{D}_q(X) \ar@{~>}[dl] \ar@{~>}[dddd]&\\
                     U(\mathfrak{g}^\dd) \ar@{~>}[dddd] \ar@{->}[rr]^{\phantom{===}\hat\mu^\#}&& \cD(X) \ar@{~>}[dddd]&&\\
                    &&&&\\
                     &&&&\\
                     &O(\mathbb{G}^\dd) \ar@{~>}[dl] \ar@{->}[rr]^{\widetilde{\mu}^\#\phantom{==}}|{\phantom{=}} &&O(T^*X^\circ) \ar@{~>}[dl]\\
                       S(\mathfrak{g}^\dd)\ar@{->}[rr]^{\mu^\#} &&O(T^*X)  &} 
\end{equation}
Each of these moment maps allows us to construct the so-called ``Hamiltonian reduction" of the target algebra with respect to a character of the source algebra.  The reduction of $O(T^*X)$ by $\mu^\#$ at a central coadjoint orbit $\lambda$ yields the algebra $\Gamma(\mathcal{M}_\dd^\lambda(Q),\cO_{\mathcal{M}^\lambda_\dd(Q)})$ of global sections of structure sheaf of the quiver variety $\mathcal{M}_\dd^\lambda(Q)$ constructed by Lusztig.  Likewise, the reduction of $O(T^*X^\circ)$ by $\widetilde{\mu}^\#$ at a central adjoint orbit $\xi$ of $\mathbb{G}^\dd$ yields the algebra $\Gamma(\widetilde{\mathcal{M}}_\dd^\xi(Q),\cO_{\widetilde{\mathcal{M}}^\xi_\dd(Q)})$ of global sections of the structure sheaf on the multiplicative quiver variety $\widetilde{\mathcal{M}}^\xi_\dd(Q)$ constructed by Crawley-Boevey and Shaw \cite{C-BS}.  The quantized Hamiltonian reduction of $\cD(X)$ by $\hat\mu^\#$ at a character of $U(\g^\dd)$ yields the a quantization $\widehat{\Gamma}(\mathcal{M}_\dd^\lambda(Q), \cO_{\mathcal{M}^\lambda_\dd(Q)})$ of the symplectic structure on $\mathcal{M}^\lambda_\dd$.  Finally, the quantum Hamiltonian reduction of $\cD_q(X)$ by $\mu^\#_q$ at a character $\xi$ of $U'_q(\g^\dd)$ is a new algebra, which we denote $A^\xi_\dd(Q)$.  We obtain the following organizing diamond, echoing \eqref{Uqdiag}, \eqref{Dqdiag}:
\begin{equation}\label{Addiag}\xymatrix@!C@C=-15pt@R=40pt{ & A_\dd^\xi(Q) \ar@{~>}[dl]_{\underset{\textrm{limit $q\to 1$}}{\textrm{\scriptsize quasi-classical }}} \ar@{~>}[dr]^{\underset{\textrm{limit $q\to 1$}}{\textrm{\scriptsize classical}}} & \\
                     \widehat{\Gamma}(\mathcal{M}_\dd^\lambda(Q), \cO_{\mathcal{M}^\lambda_\dd(Q)}) \ar@{~>}[dr]_{\underset{\textrm{graded}}{\textrm{\scriptsize associated }}} && \Gamma(\widetilde{\mathcal{M}}_\dd^\xi(Q),\cO_{\widetilde{\mathcal{M}}_\dd^\xi(Q)}) \ar@{~>}[dl]^{\underset{\textrm{degeneration}}{\textrm{\scriptsize rational }}}\\\
                     &  \Gamma(\mathcal{M}_\dd^\lambda(Q),\cO_{\mathcal{M}^\lambda_\dd(Q)}) & }
                     \end{equation}
The construction of the algebra $A^\xi_\dd(Q)$, and enumeration of its basic properties, is the primary focus of the this paper.  At the end, we discuss degenerations, and give an application to the representation theory of the spherical double affine Hecke algebra of type $A_n$, and those associated to star-shaped quivers.
\subsection{Outline of results}\label{sec:IntRes}

In Section \ref{construction}, beginning with the data of a quiver $Q$ and its dimension vector $\dd$, we construct an algebra $\cD_q=\cD_q(\Mat_\dd(Q))$.  The algebra $\cD_q$ is a braided tensor product of the algebras $\cD_q(e)$ associated to each edge $e$ of $Q$, while each $\cD_q(e)$ is a straightforward $q$-deformation of the Weyl algebra associated to the standard affine space $\Mat(d_\alpha,d_\beta)$.  The relations of $\cD_q(e)$ are given in a formal way designed to make their equivariance properties evident; the reader interested in a hands-on, RTT-type presentation can skip ahead to Sections \ref{RTTsec1} and \ref{RTTsec2}.

Our first theorem is Theorem \ref{flatness}, which states that the algebra $\cD_q$ is a flat $q$-deformation of the algebra $\cD(\operatorname{Mat}_\dd(Q))$ of differential operators on the matrix space of the quiver $Q$. Our proof is modeled on Theorem 1.5 of \cite{GZ}, and consists of constructing an explicit PBW basis of ordered monomials, which clearly deforms the usual basis of $\cD(\Mat_\dd(Q))$.  

The defining relations for $\cD_q$ in examples related to quantum groups are similar to the $FRT$-construction of quantum coordinate algebras, and are also closely related to the algebras $\cD_q(GL_N)$ of quantum differential operators on $GL_N$, which have been studied by many authors.  In Sections \ref{RTTsec1} and \ref{RTTsec2}, we list out the relations in detail for these examples of interest, and explain their relation to known constructions.

The algebra $\cD_q$ possesses certain elements $\operatorname{det}_q(e)$, for each edge $e\in E$, which conjugate standard monomials in $\cD_q$ by  powers of $q$ (the proof of this assertion is delayed until Section 4, Corollaries \ref{qcentral} and \ref{oreset}).  We therefore localize $\cD_q$ at the multiplicative Ore set generated by these $q$-determinants, to obtain an algebra $\cD_q^\circ$, in which certain quantum matrices become invertible.

In Section \ref{q-Fourier}, we construct a $q$-analog, $\mathcal{F}$, of the classical Fourier transform map on the algebra $\cD(\Mat_\dd(Q))$, which allows us to prove the independence of $\cD^\circ_q(\Mat_\dd(Q))$ on the orientation of $Q$.  Our main results in this section are Definition-Propositions \ref{gencasenoloop} and \ref{gencaseloop}, where $\mathcal{F}$ is defined explicitly on generators; necessary relations are checked directly.  As a warmup, we work out one-dimensional examples in Definition-Propositions \ref{easycasenoloop} and \ref{easycaseloop}, whose proofs foreshadow the general one.  


In Section \ref{q-moment}, we define a $q$-deformed, braided analog $\mu_q^\#$ of the multiplicative moment map underlying relation \eqref{mdppa}.  We subsequently define an analog of Hamiltonian reduction in this context, which is closely related to Lu's notion \cite{L} for Hopf algebras, and is also inspired by the quantum moment maps appearing in \cite{VV}.  The output of this Hamiltonian reduction is an algebra $A_\dd^\lambda(Q)$, which $q$-quantizes the space $\mathcal{M}_\dd^\lambda$.  The main results of Section 6 are Definition-Propositions \ref{momentalpha} and \ref{momentbeta}, and Propositions \ref{mommap} and \ref{mommaploop}, in which the moment map is defined, and the moment map condition is verified.

In Section \ref{degeneration}, we consider relations between the algebra $A^\lambda_\dd(Q)$ and well-known constructions in representation theory - specifically quiver varieties and spherical double affine Hecke algebras.  To begin, we study flatness properties of $A^\lambda_\dd(Q)$ as the parameter $q$ varies.  While the flatness of the algebra $\cD_q$ is proven directly, the flatness of the algebra $A_\dd^\lambda(Q)$ is considerably more subtle.  This is because the argument we give for $\cD_q$ relies upon the existence of a $\ZZ$-grading with finite dimensional graded components; this grading does not descend to $A_\dd^\lambda(Q)$.

For this reason, we restrict ourselves to situations where the classical moment map $\mu$ is flat (as in Theorem \ref{CBflat}), and we consider the question of formal flatness of  $A_\dd^\lambda(Q)$.  That is, we set $q=e^\hbar$, and consider the algebras $\cD_q$ and $U_q(\g^\dd)$, moment map $\mu_q$, and Hamiltonian reduction $A_\dd^\lambda(Q)$ all in the category of $\CC[[\hbar]]$-modules.   We prove that   $A_\dd^\lambda(Q)$ is a topologically free $\CC[[\hbar]]$-module, so that the deformation is flat in the formal neighborhood of $q=1$. 

Next, we address the question: what algebra does $A_\dd^\lambda(Q)$ deform?  In answering this question, we must explain that there is a unifier in the construction of $\mu_q^\#$, and in its simultaneous relation to (classical, multiplicative, and quantized) moment maps $\mu$.  Recall the two variations of Hamiltonian reduction in classical geometry: ``quantized Hamiltonian reduction" and ``quasi-Hamiltonian reduction".  In the former, the moment map is a homomorphism of algebras $$\widehat{\mu}^\#: U(\g^\dd)\to \cD(\Mat_\dd(Q)),$$ while in the latter, we have a morphism of varieties $\widetilde{\mu}: T^*\Mat_\dd(Q)^\circ \to \mathbb{G}^\dd$, or equivalently a map of algebras $$\widetilde{\mu}^\#:\cO(\mathbb{G}^\dd)\to \cO(T^*\Mat_\dd(Q)^\circ).$$  In classical geometry, there are analogies between these moment maps, but not a precise connection.  We will see that the map $\mu_q^\#$ bears a precise relationship to both maps $\widetilde{\mu}^\#$ and $\widehat{\mu}^\#$, under degeneration.
 
Recall that the Hopf algebra $U=U_q(\g^\dd)$ has a large co-ideal subalgebra $U'$, consisting of the elements which are locally finite under the adjoint action of $U$ on itself (see \cite{JL} for details, and for the sense in which $U'$ is ``large").  The homomorphism $\mu_q^\#$ maps out of $U'$, and is a $q$-deformed quantum moment map, as considered by Lu \cite{L} and \cite{VV}.  On the other hand, we have Majid's covariantized coordinate algebra $A_q(\mathbb{G}^\dd)$, a flat deformation of $O(\mathbb{G}^\dd)$, and we have the Rosso isomorphism $\kappa: A_q(G)\overset{\sim}{\rightarrow} U'$ (see \cite{Ma}).  Thus, we may also view $\mu_q^\#$ as a quantization of the group-valued moment map underlying equation \eqref{mdppa}.  We summarize these relationships in the following diagram:
$$\xymatrix{ \cO(\mathbb{G}^\dd)  \ar@{->}[d]_{\mu^\#}& \ar@{~>}[l]^{q\to 1} A_q(\mathbb{G}^\dd) \ar@{<->}[r]_{\kappa}^{\sim} \ar@{->}[d]_{\mu^\#_q}& U'_q(\g^\dd) \ar@{->}[d]_{\mu^\#_q} \ar@{~>}[r]_{q\to 1} & U(\g) \ar@{->}[d]_{\widehat{\mu}}
 \\ \cO(\mathrm{Mat}_\dd(Q)^\circ)&\ar@{~>}[l]^{\,\,\,\,\,\,\,\,\,\,\,\,q\to 1} \cD_q \ar@{=}[r]& \cD_q \ar@{~>}[r]_{q\to 1\,\,\,\,\,\,\,\,}& \cD(\mathrm{Mat}_\dd(\overline{Q}))}
$$
Thus, taking quasi-Hamiltonian reduction along $\mu^\#$, $q$-deformed quantum Hamiltonian reduction along $\mu_q^\#$, and quantum Hamiltonian reduction along $\widehat{\mu}$, we have the ``commutative diagram" \eqref{Addiag} of deformations and degenerations of the corresponding Hamiltonian reductions.

As an application, we show in Theorem \ref{AnDAHA} that the algebra $A_\dd^\lambda(Q)$ is isomorphic to the spherical DAHA of type $A_n$, when $Q$ and $d$ are the Calogero-Moser quiver and dimension vector:
$$(Q,d)= \overset{1}{\bullet}\rightarrow\overset{n}{\bullet}\rightloop,$$
which allows us to give a new description of the representation category of the spherical DAHA as a quotient of the category of equivariant $\cD_q(\Mat_\dd(Q))$-modules by a certain Serre subcategory of aspherical modules.  This assertion follows from generalities about flat deformations, together with the fact that the spherical DAHA is the universal deformation of the corresponding rational Cherednik algebra, which itself may be built by quantum Hamiltonian reduction from $\widehat{\Gamma}(\mathcal{M}_\dd^\lambda(Q), \cO_{\mathcal{M}^\lambda_\dd(Q)})$.  In fact, because we have restricted to formal parameters $q=e^\hbar$, this result is not very valuable, as the spherical DAHA is actually a trivial deformation over $\CC[[\hbar]]$ of the spherical rational Cherednik algebra.  However, we expect this isomorphism to hold also numerically, for generic $q$ and $\lambda$.

Likewise, if $Q$ is a so-called star-shaped quiver (meaning all vertices are uni- or bi-valent, except for a single vertex, called the node) we have an isomorphism between $A_\dd^\lambda(Q)$ and the generalized spherical DAHA, defined in \cite{EOR}.  Star-shaped quivers play a central role in the approach to the Deligne-Simpson problem in \cite{C-B2}, \cite{C-BS}.  These applications suggest that the algebras $A_\dd^\lambda(Q)$ may be viewed as further generalizations of the (spherical) DAHA, to an arbitrary quiver $Q$.

\subsection{Future directions}

One motivation of this work is to extend the unifying structure of quivers to a myriad of constructions in ``quantum" algebraic geometry, such as quantum planes, and their $q$-Weyl algebras, FRT algebras, reflection equation algebras, differential operator algebras on quantum groups, and perhaps most importantly, double affine Hecke algebras.

In particular, for certain explicit quivers and parameters, we are able to identify our algebras $A_\dd^\xi(Q)$ with spherical double affine Hecke algebras of type $A$, and also with generalized double affine Hecke algebras associated to star-shaped quivers (see Example \ref{starexample}).  We anticipate relations with Gan-Ginzburg algebras \cite{GG1}, as $q$-deformations of Montarani's constructions \cite{Mo}, which relate symplectic reflection algebras to $D$-modules on quivers.  Given the ubiquity of quiver formalism in Lie theory, we expect to be able to apply our deformation procedure in a wide class of examples to obtain deformations of Lie theoretic objects.

When the quantum parameter $q$ is a non-trivial root of unity, the algebras $\cD_q$ and their reductions $A^\xi_\dd(Q)$ obtain large ``$q$-centers;'' in simple examples, $\cD_q$ is Azumaya over its $q$-center, which turns out to be identified with $O(T^*X^\circ)$.  In such examples, the algebra $A^\xi_\dd(Q)$ is Azumaya over its $q$-center, which is identified with the corresponding multiplicative quiver variety.  We intend to explore this property in general, and thus obtain a second ``quasi-classical" degeneration of our algebra, analogous to the well-known ``$p$-center" phenomenon for varieties in characteristic $p$.

More generally, we hope to combine the techniques of \cite{J} with the new examples constucted herein to produce representations of braid groups of punctured, Riemann surfaces of higher genus.

\subsection{Acknowledgments}
I am very grateful to Pavel Etingof and Kobi Kremnizer, for proposing this line of research, and for innumerable helpful conversations and suggestions throughout my time at MIT; this project would have proceeded nowhere without their guidance.  In particular, the definition of $\cD_q(e)$ for non-loops $e$ contained herein was essentially proposed to me by Kobi Kremnizer, who also conjectured a flatness result along the lines of Theorem \ref{flatness}.  I am grateful to Pavel Etingof for patient explanations about the flatness and degeneration arguments in Section \ref{flatness-sec}. I would also like to acknowledge helpful discussions with Adrien Brochier and Damien Calaque, during which a variation of the homomorphism of Definition-Proposition \ref{gencaseloop} was discovered.  I am grateful to Ivan Losev for a careful reading of a previous draft, which has greatly improved the exposition.

\printnotation

\section{Preliminaries}\label{prelims}
\subsection{Moduli spaces of quiver representations}

Let \glossary{$Q$,$V$,$E$,$\alpha$,$\beta$}$Q=(V,E)$ denote a connected quiver, with vertex set $V$, and directed edge set $E$.  For $e\in E$, let $\alpha=\alpha(e)$ and $\beta=\beta(e)$ denote the tail and head of $e$, respectively.  The subject of much study is the category $\Rep\, Q$, of representations of $Q$.  An \glossary{$\Rep\,Q$}object $X$ of $\Rep\,Q$ is an assignment of a finite dimensional vector space $X_v$ over $\CC$ to each $v\in V$, and a linear operator $X_e: X_{\alpha}\to X_{\beta}$ to each $e \in E$.  A morphism $\phi$ between $X$ and $Y$ is a collection of linear maps $X_v\to Y_v$, which satisfy $Y_e\circ\phi_\alpha=\phi_\beta\circ X_e$, for all $e\in E$.

Fix a function $\dd:V\to\ZZ_{\geq 0}, v\mapsto d_v$ (called a dimension vector), and consider the following affine variety and affine algebraic group, respectively:

\glossary{$\Mat_{\dd}(Q)$,$\mathbb{G}^\dd$}$$\Mat_{\dd}(Q):=\prod_{e\in E} \Mat(\CC^{d_\alpha},\CC^{d_\beta}), \,\,\, \mathbb{G}^\dd:=\prod_{v\in V} GL(\CC^{d_v}).$$
We let $\mathbb{G}^\dd$ act on $M\in\Mat_\dd(Q)$ by change of basis at each vertex,
$$(g.M)_e:= g_{\beta(e)} M_e g_{\alpha(e)}^{-1}.$$



Many important applications of the representation theory of quivers involve the doubled quiver \glossary{$\overline{Q}$,$\overline{E}$,$ E^\vee$}$\overline{Q}=(V,\overline{E}=E\cup E^\vee)$, built from $Q$ by adding an adjoint arrow \mbox{$\beta(e)\xrightarrow{e^\vee} \alpha(e)\in E^\vee$}, for each $e\in E$.  We have canonical isomorphisms,
$$ T^*\Mat(\CC^{d_\alpha},\CC^{d_\beta})\cong \Mat(\CC^{d_\alpha},\CC^{d_\beta}) \times \Mat(\CC^{d_\beta},\CC^{d_\alpha}),$$ with the standard symplectic pairing given by:
$$(M,N)=\tr (M_eN_{e^\vee}-N_eM_{e^\vee}).$$
Taken together, these give an identification $T^*\Mat_\dd Q\cong\Mat_\dd \overline{Q}$.  Clearly, $\mathbb{G}^\dd$ acts by symplectomorphisms; moreover, the action admits a moment map:
$$\glossary{$\mu$}\mu: \Mat_\dd(\overline{Q})\to \g^\dd,$$
$$M\mapsto\sum_{e\in E}[M_e,M_{e^\vee}],$$
where we set $\g^\dd:=\operatorname{Lie}(\mathbb{G}^\dd)$.  Thus we may construct the Hamiltonian reduction along $\mu^{-1}(0)$:
$$\mathcal{M}_\dd(Q):= \Mat_\dd(\overline{Q})\dS_{\mu,0} \mathbb{G}^\dd,$$
a Poisson affine algebraic variety.  That is, we first impose the condition on $M\in\Mat_\dd(\overline{Q})$ that:
\begin{equation}\label{ppa}\sum_{e\in E} [M_e, M_{e^\vee}] = 0,\end{equation}
and we then take the categorical quotient  of the subvariety of such $M$ by the action of $\mathbb{G}^\dd$.  On the level of coordinate functions, we have:
$$\cO(\mathcal{M}_\dd(Q)):=\left(\mathcal{A} \Big / \mathcal{A}\mu^\#(\g^\dd)\right)^{\mathbb{G}^\dd},$$
where $\mathcal{A}=\cO(T^*\Mat_\dd(Q))$.

\subsection{Deformed pre-projective algebras}
The \emph{preprojective algebra}, $\Pi_0(Q)$ of $Q$ \cite{GP}, is the quotient of the path algebra $\CC\overline{Q}$ by relation,
$$\sum_{e\in E} [e,e^\vee] =0,$$
corresponding to equation \eqref{ppa}.  The variety $\mathcal{M}^{\mathbf{0}}_\dd(Q)$ may thus be interpreted as a moduli space of $\dd$-dimensional semi-simple representations of $\Pi_0(Q)$.  More generally, given a vector $\lambda: V\to \CC$, we may construct the Hamiltonian reduction $\mathcal{M}_\dd^\lambda$ along $\mu^{-1}(\sum \lambda_v \id_v)$.  That is, we first impose the condition on $M\in \Mat_\dd(Q)$ that:
\begin{equation}\sum_{e\in E} [M_e,M_{e^{\vee}}] = \sum_{v}\lambda_v\id_v,\label{dppa}\end{equation}
and then take the categorical quotient of the subvariety of such $M$ by the action of $\mathbb{G}^\dd$.  We assume that $\lambda\cdot d=0$, as otherwise equation \eqref{dppa} implies that $\mathcal{M}_\dd^\lambda$ is empty. The \emph{deformed pre-projective algebras}, $\Pi_\lambda(Q)$, were constructed by Crawley-Boevey and Holland in \cite{C-BH}, and have since received wide attention. These algebras are quotients of the path algebra $\CC\overline{Q}$ by the relation
$$ \sum_{e\in E} [e,e^\vee] = \sum_{v\in V} \lambda_v \iota_v,$$
corresponding to equation \eqref{dppa}.  The variety $\mathcal{M}_\dd^\lambda$ may be interpreted as a moduli space of semi-simple representations of $\Pi_\lambda(Q)$.

In the present work, we will be concerned with certain flat non-commutative deformations of the variety $\mathcal{M}_\dd^\lambda(Q)$.  The flatness of our deformations depends, in turn, on the flatness of the classical moment map $\mu$.  Fortunately, there is a completely explicit criterion for the flatness of $\mu$, due to Crawley-Boevey.  Let $A$ denote the Cartan matrix associated to $Q$, and let $p:\ZZ^V\to\CC$ denote the function:
$$p(d):= 1 - \frac12(d,Ad) =  1+ \sum_{e\in E}d_{\alpha(e)}d_{\beta(e)} - \sum_{v\in V}d_v^2.$$
We have:
\begin{theorem}\cite{C-B1}\label{CBflat} The following are equivalent:
\begin{enumerate}
\item $\mu$ is a flat morphism of algebraic varieties.
\item $\mu^{-1}(0)$ has dimension $(d,d)-1+2p(d)$.
\item $p(d)\geq \sum_ip(r_i)$, for any decomposition $d=\sum_ir_i$ into positive roots $r_i$.
\end{enumerate}
\end{theorem}

Moreover, if it happens that $d$ satisfies the strict inequality in (3) for all possible non-trivial decompositions $d=\sum_ir_i$ into positive roots $r_i$, then it is shown in \cite{C-BEG}, Theorem 11.3.1, that the fibers, $\mu^{-1}(\sum_{v}\lambda_v\id_v)$, are all reduced and irreducible complete intersections.  In this case, $\mathcal{M}_\dd^\lambda$ coincides with its smooth resolution for generic $\lambda$, and in particular, both are actually affine.

For Dynkin quivers $Q$, Theorem \ref{CBflat} asserts that $\mu$ is flat if and only if $\dd$ is a positive root for $Q$; in this case the classical Hamiltonian reduction is zero-dimensional, so these are not interesting examples from the point of view of deformation theory.

For affine Dynkin quivers $Q$, let $\delta$ denote the positive generator of the imaginary root lattice.  In this case, Theorem \ref{CBflat} asserts that $\mu$ is flat in one of two cases:  when $\dd=r_i + \delta$, for a root $r_i$ of the ordinary Dynkin quiver associated to $Q$, or when $\dd=\delta$.  In the former case, the classical Hamiltonian reduction is zero-dimensional, while in the latter case it is two-dimensional, and gives the Kleinian singularity associated with $Q$.

The most interesting examples come from quivers $Q$, which are neither of Dynkin nor affine-Dynkin type.  For such quivers, it is shown in \cite{C-BEG}, Lemma 11.3.3, that the strict version of condition (3) above is satisfied by a Zariski-dense set $\Sigma_0$ of dimension vectors $\dd\in \mathbb{Z}_{\geq 0}^V$.  Thus, such $Q$ produce a rich family of examples of flat Hamiltonian reductions of positive dimension.  Of particular interest are the so-called ``Calogero-Moser" quivers obtained by adding a ``base" vertex $\tilde{v}$ to an affine Dynkin quiver, whose unique edge connects it to the extending vertex.  In this case, the dimension vector $n\delta + \tilde{v}$ satisfies the strict version of condition (3) in Theorem \ref{CBflat} for any $n \geq 0$.

\subsection{Multiplicative deformed pre-projective algebras}
The deformed pre-projective algebra admits a multiplicative deformation, which may be described as follows.  Extend $e\mapsto e^\vee$ to an involution on $\overline{E}$, by setting $e^{\vee\vee}:=e$, and define $\epsilon(e)=1$, if $e\in E$, $-1$ else.  We choose an ordering on the edges $e\in \overline{E}$, and a function $\xi:V\to\CC^\times$.  First, we restrict our attention to the set $T^*\Mat_\dd(Q)^\circ$ of $X\in T^*\Mat_\dd(Q)\cong\Mat_\dd(\overline{Q})$ such that, for each $e\in \overline{E}$, the matrices $(\id_\alpha + X_{e^\vee}X_e)$ are invertible.  Further, we impose the following restriction, which is a multiplicative version of equation \eqref{dppa}\footnote{here, $\overrightarrow{\prod}$ denotes ordered product; see Section \ref{sec:QuivNot}.}:
\begin{equation}\label{mdppa}\overrightarrow{\prod_{e\in \overline{E}}} (\id_\alpha + M_{e^\vee}M_e)^{\epsilon(e)}= \sum \xi_v\id_v.\end{equation}
Taking the categorical quotient by the action of $\mathbb{G}^\dd$, we obtain the space
$\widetilde{\mathcal{M}}_\dd^\xi$, which again has an interpretation as moduli of $\dd$-dimensional semi-simple representations for a certain localization of $\CC\overline{Q}$, known as the multiplicative deformed pre-projective algebra.  As has been noted in \cite{C-BS} and \cite{VdB1,VdB2}, the variety $\widehat{\mathcal{M}}_\dd^\xi$ is in fact an instance of quasi-Hamiltonian reduction along a ``group-valued'' moment map,\glossary{$\widetilde{\mu}$}
$$\widetilde{\mu}:T^*\Mat_\dd(Q)^\circ\to\mathbb{G}^\dd$$
$$ X \mapsto \overrightarrow{\prod_{e\in E}}(\id_\alpha + X_{e^\vee}X_e)^{\epsilon(e)},$$
with respect to the central element $(\xi_v\id_v)_{v\in V}\in\mathbb{G}^\dd$.  Quasi-Hamiltonian reduction is a multiplicative analog of Hamiltonian reduction, as defined by Alekseev and Kosmann-Schwarzbach in \cite{AK-S}.  See also \cite{AMM}, \cite{AK-SM}, for foundational development of quasi-Poisson geometry, and group-valued moment maps.

\subsection{Quantized quiver varieties}
Finally, there is a quantization of the variety $\mathcal{M}_\dd^\lambda$ \cite{GG2}, which involves replacing the cotangent bundle to $\Mat_\dd(Q)$ with its quantization, the algebra $\cD(\Mat_\dd(Q))$ of differential operators.  The algebra $\cD(\Mat_\dd(Q))$ quantizes the symplectic form on $T^*\Mat_\dd(Q)$, and one constructs its quantum Hamiltonian reduction $\widehat{\Gamma}(\mathcal{M}_\dd^\lambda(Q), \cO_{\mathcal{M}^\lambda_\dd(Q)})$ along the homomorphism,\glossary{$\widehat{\mu}^\#$}
$$\widehat{\mu}^\#:U(\g^\dd)\to \cD(\Mat_\dd(Q)),$$
$$ X\in \g \mapsto L_X,$$
where $L_X$ is the vector field generated by the action of $X$.  The reduction is taken with repsect to a character $\lambda:U(\g^\dd)\to\CC$.  For an exposition of quantum Hamiltonian reduction, see \cite{L}.  For some applications, see \cite{E1}, \cite{Mo}.  \edit{We hope to explain these constructions in an appendix to a future version of this paper.}

\subsection{The multiplicative Deligne-Simpson problem}

The applications of quiver varieties and (multiplicative, deformed) pre-projective algebras to diverse areas of mathematics are too many to list here; as such we mention only one important application, due to Crawley-Boevey and Shaw \cite{C-BS}.  Given conjugacy classes $C_1,\ldots,C_n\subset GL(V)$, the Deligne-Simpson problem asks when there exists an irreducible local system on $\mathbb{P}^1\backslash\{p_1,\ldots p_n\}$ with monodromy around each $p_i$ given by a matrix $A_i\in C_i$.  Thus, the Deligne-Simpson problem concerns the classification of $n$-tuples of matrices $A_i\in C_i$ without common fixed subspace, satisfying:
$$A_1\cdots A_n=\id,$$
up to simultaneous conjugation of the $A_i$.

Crawley-Boevey and Shaw were able to answer this question rather concretely in terms of the root data of a certain star-shaped quiver $Q$, which encodes the conjugacy classes $C_i$.  They determine for which $Q$, with the relevant dimension vector $\dd$, the variety $\widehat{\mathcal{M}_\dd^q}$ is non-empty and, in this case, what is its dimension.  Still, the finer geometry of these varieties is not completely well-understood.  The connection between multiplicative quiver varieties and fundamental groups of Riemann surfaces is a major motivation for the present work.

In particular, there is a well-known symplectic structure on the space of bundles with flat connections on a compact, closed oriented two-manifold with boundary of genus $g$.  A quantization of this symplectic structure has been considered in \cite{FR}, and constructed in \cite{RS}; our results provide another construction, and a generalization to arbitrary quivers.

\section{Construction of $\cO_q(\Mat_\dd(Q))$ and $\cD_q(\Mat_\dd(Q))$}\label{construction}
\subsection{Discussion}\label{discussion}
The constructions in this section are phrased in the language of braided tensor categories, while all that is essential for our primary example is a vector space $V$, the tensor flip $
\tau:v\ot w\mapsto w\ot v$, and a Hecke R-matrix, $R:V\ot V\to V\ot V$, satisfying the ``quantum Yang Baxter" equation (QYBE), \glossary{QYBE, Hecke reln.}
$$\tau_{12}R_{12}\tau_{23}R_{23}\tau_{12}R_{12} = \tau_{23}R_{23}\tau_{12}R_{12}\tau_{23}R_{23}:V\ot V\ot V \to V\ot V\ot V,$$
and the quadratic ``Hecke" relations:
$$\tau\circ R - R^{-1}\circ \tau = (q-q^{-1})\id\ot\id.$$
There are nevertheless several practical reasons for adopting the tensor categorical formalism over the more concrete data of Hecke $R$-matrices.

First, when deforming algebras with geometrical significance, it is often not clear at the outset precisely how to proceed: the set of ``bad" definitions is open dense in the space of all possible definitions.  That is, given only the goal of producing some new algebra with similar generators and relations, which ``degenerates" to the classical algebra when $q\to 1$, there is far too much flexibility, and many pathologies can arise (as regards flatness, zero-divisors, localizations, etc.).  However, in the present work, we require that our algebras $\cD_q(\Mat_\dd(Q))$ enjoy the following properties:
\begin{enumerate}
\item $\cD_q(\Mat_\dd(Q))$ is a algebraically flat deformation of $\cD(\Mat_\dd(Q))$.  This means we exhibit an explicit PBW-basis for $\cD_q(\Mat_\dd(Q))$ specializing to the standard monomials when $q=1$.  This condition is much stronger than being formally flat.
\item  $\cD_q(\Mat_\dd(Q))$ carries an action of the quantum group $U_q(\g^\dd)$, which quantizes $\mathbb{G}^\dd$:\glossary{$U_q(\g^\dd)$}
$$U_q(\g^\dd):=\bigotimes_{v\in V} U_q(\mathfrak{gl}_{d_v}).$$
\item There exists a ``quantum moment map" $\mu_q^\#$, simultaneously quantizing and $q$-deforming the classical moment map $\mu$.\end{enumerate}
Requirements (1) and (2) suggest that the algebra $\cD_q(\Mat_\dd(Q))$ necessarily is an algebra in the braided tensor category $\cC=U_q(\g^\dd)$-lfmod of locally finite modules for $U_q(\g^\dd)$.  This drastically restricts which sorts of algebras we may consider, namely to those whose generators and relations express as the image of morphisms in the braided tensor category $\cC$.

Secondly, in condition (3), we require the moment map itself to be equivariant for the quantum group, which means that it is a homomorphism of algebras in $\cC$.  Since we wish the construction to be uniform for different dimension vectors $\dd$, it is natural to allow ourselves only the axioms of a braided tensor category, together with the Hecke relation on the braiding.  This turns out to be a useful restriction, as it narrows our focus sufficiently such that the ``right" definitions are essentially the only ones we are able to write down.

A third practical benefit from working with braided tensor categories has already surfaced in \cite{J}, \cite{JM}, where we studied interplay between certain algebraic constructions in Lie theory and geometry of spaces of configurations of points on Riemann surfaces.  These constructions are greatly clarified by the use of braided tensor categorical language and quantum groups, in the same way that the language of braided tensor categories clarifies the connections between quantum groups and knot invariants.

In addition to the practical motivations above, there are two more substantive motivations for working with braided tensor categories.  The first is that there are more braided tensor categories besides $U_q(\g^\dd)$-lfmod that we can associate to $Q$.  Two particularly tantalizing examples are:
\begin{enumerate}
\item Fusion categories associated to quantum groups at roots of unity
\item Deligne's categories $U_q(\g_\nu)$-mod, where $\nu:V\to\CC$ has as values arbitrary complex numbers, rather than positive integers.
\end{enumerate}
We hope that the methods of this paper will go through in these settings more or less intact, which would open the door for connections to modular categories and invariants of links and knots on higher genus surfaces.  The second motivation is related to the notion of a quasi-symmetric tensor category, which is a braided tensor category over $\CC[[\hbar]]$ such that the braiding satisfies:
$$\sigma_{W,V}\sigma_{V,W}= \id_{V\ot W} \mod \hbar.$$
It is well-known how to degenerate such categories into symmetric tensor categories.  In case $\cC$ is the representation category of a quantum group $U_q(\mathfrak{g})$, the first-order term in $\hbar$ often carries some interesting data for Lie theory: for instance, the first non-trivial term of $\sigma_{W,V}\sigma_{V,W}$ is essentially the Casimir operator $\Omega\in\operatorname{Sym}^2(\g)^\g$, while the first non-trivial term of the associator is the unique invariant alternating 3-form, $\phi\in\Lambda^3(\g)^\g$.

As an application of these ideas one can recover the axioms of quasi-Poisson geometry as first-order degenerations of the axioms for algebras in braided tensor categories.  It is our hope that the axioms of ``group-valued moment maps" can also be obtained as degenerations of the notion of quantum Hamiltonian reduction.

\subsection{Reminders on braided tensor categories}
In this section, we recall some basic constructions involving braided tensor categories, in order to fix notations.  As such, we do not discuss all details, but only those we will use explicitly.  For clarity's sake, we supress instances of the associativity and unit isomorphisms in definitions and commutative diagrams, as they can be inserted uniquely, if necessary.

Recall that a tensor category is a $\CC$-linear abelian category $\cD$, together with a biadditive functor, $$\ot: \cD\times \cD\to \cD,$$
linear on $\Hom$'s, together with a unit $\mathbf{1}\in\cC$, associativity isomorphism $\alpha$, and unit isomorphisms.  These are required to satisfy a well-known list of axioms, which we do not recall here.  A tensor functor $F=(F,J)$ between tensor categories $\cD_1$ and $\cD_2$ is an exact functor $F:\cD_1\to\cD_2$ of underlying abelian categories, together with a functorial isomorphism,
$$J:F(-)\ot F(-)\to F(-\ot -),$$
respecting units and associators in the appropriate sense.  The opposite tensor category $\cD^{\vee}$ is the same underlying abelian category, with tensor product \mbox{$V\ot^{op} W:=W\ot V$}, and associator $\alpha^{-1}$.  A braided tensor category is a tensor category $\cD$, together with a natural isomorphism $\sigma:\ot\to\ot^{op},$ satisfying the so-called hexagon relations.

\subsubsection{Deligne's external product of abelian categories}\label{extprod}
Recall that a $\CC$-linear abelian category $\cD$ is called \emph{locally finite}, if all $\Hom$ spaces are finite dimensional, and every object $V\in\cD$ has finite length.  \glossary{$\bt$}We use the symbol $\bt$ to denote Deligne's tensor product of locally finite categories (see, e.g. 
\cite{EGNO}).  In this article, we will consider semisimple abelian categories; in this case, the external tensor product $\cD_1\bt\cD_2$ of $\cD_1$ and $\cD_2$ is just a semisimple abelian category with simple objects $X\bt Y$, where $X$ and $Y$ are simples in $\cD_1$, $\cD_2$.  External tensor products may be defined for non-semisimple categories - this will be needed when considering $q$ to be a root of unity - but we will not need them here.

\begin{example}
Let $A$ be a (possibly infinite-dimensional) $\CC$-algebra.  Then the category $A$-fmod of finite dimensional $A$-modules is a locally finite $\CC$-linear abelian category.  For two such algebras $A$ and $B$, we have a natural equivalence,
$$A\textrm{-fmod}\bt B\textrm{-fmod} \sim (A\ot B)\textrm{-fmod}.$$
In all our examples the external tensor products of categories we consider are of this sort.
\end{example}

The Deligne tensor product $\cD_1\bt\cdots\bt\cD_n$ of (braided) locally finite tensor categories is again a (braided) locally finite tensor category, with structure functors defined diagonally: we set  $\ot:=\ot_1\bt\cdots\bt\ot_n$ (and $\sigma:=\sigma_1\bt\cdots\bt\sigma_n$).

\subsection{Primary objects}
In this section we construct the algebras $\cO_q(\Mat_\dd(Q))$ and $\cD_q(\Mat_\dd(Q))$ as algebras in a braided tensor category $\cC$ associated to $Q$.
\subsubsection{Quiver notation}\label{sec:QuivNot}
We resume the notation for quivers from the Introduction.  We choose, once and for all, an ordering on $\overline{E}=E\cup E^{\vee}$: we will emphasize dependence on this ordering in later definitions with an over-arrow decoration, e.g \glossary{$\overrightarrow{\otimes}, \overrightarrow{\prod}$}$\overrightarrow{\otimes}, \overrightarrow{\prod}$.  For $v\in V$, we define \glossary{$E^\alpha_v$,$E^\beta_v$, $E_v^\circ$}$E^\alpha_v$ and $E^\beta_v$ as the subsets of non-loop edges $e\in E$ such that $\alpha(e)=v$ or $\beta(e)=v$, respectively; we define $E_v^\circ$ as the subset of self-loops based at $v$.  Each obtains an induced ordering from $E$.

For each $v\in V$, we fix a locally finite braided tensor category $\cC_v$, and a distinguished object $W_v\in\cC_v$.  This data we encode in the function $\dd$, by defining $\dd(v):=(W_v,\cC_v)$.  Thus, $\dd$ is a generalized dimension vector, specifying which object is associated to each vertex.

\begin{definition}
We let $\cC:=\bigboxtimes{v\in V}\cC_v$, with tensor product and braiding defined diagonally.  We regard any object $X_v\in \cC_v$ as an object in $\cC$ by putting the tensor unit $\mathbf{1}_w:=\mathbf{1}_{\cC_w}$ in the omitted tensor components. Strictly speaking, $\cC$ depends on an implicit choice of ordering on $V$; however the categories associated to different orderings are canonically equivalent by the obvious functors of transposition of factors; it is the ordering on edges which is more significant in these constructions.
\end{definition}

\subsubsection{Defining relations}
The defining relations for the algebras $\cO_q(\Mat_\dd(Q))$ and $\cD_q(\Mat_\dd(Q))$ are most naturally expressed as the image of certain canonical morphisms built from the braiding.  We define those morphisms here for use later.  For $e\in \overline{E}$, we let $\operatorname{Mat}(e):=W_\alpha^*\ot W_\beta \in \cC$\glossary{$\Mat(e)$}.  Choose a parameter $t\in\CC$.
\begin{definition} 
For $e\in E$ with $\alpha(e)\neq\beta(e)$, we define:
$$R(e):  W_\alpha^*\ot W_\alpha^*\bt W_\beta\ot W_\beta \to \operatorname{Mat}(e)\ot  \operatorname{Mat}(e),$$
$$R(e):=\sigma_{W_\alpha^*,W_\alpha^*} - \sigma_{W_\beta,W_\beta}.$$
$${S}(e,e^\vee):W_\alpha^*\ot W_\alpha \bt W_\beta^* \ot W_\beta\to \Mat(e^\vee)\ot\Mat(e)\oplus \Mat(e)\ot \Mat(e^\vee) \oplus \CC,$$
$$S(e,e^\vee):= \sigma_{W_\alpha^*,W_\alpha} - \sigma^{-1}_{W_\beta,W_\beta^*} - t\cdot(\ev_{W_\alpha}\bt\ev_{W_\beta}),$$
where $\alpha=\alpha(e)=\beta(e^\vee)$, $\beta=\beta(e)=\alpha(e^\vee)$.\glossary{$t$}
\end{definition}
\begin{definition} For $e\in E$ with $\alpha(e)=\beta(e)$, we define:
$$R(e):  W_\alpha^*\ot W_\alpha^*\ot W_\alpha\ot W_\alpha \to \operatorname{Mat}(e)\ot  \operatorname{Mat}(e),$$
$$R(e):=\sigma_{W_\alpha,W_\alpha^*}^{-1}\circ(\sigma_{W_\alpha^*,W_\alpha^*} - \sigma_{W_\alpha,W_\alpha}).$$
$$S(e,e^\vee):W_\alpha^*\ot W_\alpha^* \ot W_\alpha \ot W_\alpha\to \Mat(e)\ot \Mat(e^\vee) \oplus \Mat(e^\vee)\ot\Mat(e),$$
$$S(e,e^\vee):= \sigma^{-1}_{W_\alpha,W_\alpha^*}(\sigma_{W_\alpha,W_\alpha} - \sigma^{-1}_{W_\alpha^*,W_\alpha^*}).$$
\end{definition}

Recall that, for any object $X$ in a tensor category $\cD$, the tensor algebra,
$$T(X):=\bigoplus_{k=0}^{\infty} X^{\ot k}$$ is an naturally an algebra in $\cD$.  Given $Y\subset T(X)$, we let $\langle Y \rangle$ denote the two-sided ideal in $T(X)$ generated by $Y$.

\begin{definition} For $e\in E$, we define the two-sided ideals:\glossary{$I(e),I(e,e^\vee)$}
$$I(e):=\langle\operatorname{Im}R(e)\rangle \subset T(\Mat(e)),\,\, \textrm{and}$$ $$I(e,e^\vee):=\langle\operatorname{Im}S(e,e^\vee)\rangle\subset T(\Mat(e)\oplus\Mat(e^\vee)).$$
\end{definition}

\subsubsection{The braided coordinate and differential operator algebras of $Q$}
\begin{definition} The edge coordinate algebra $\cO_q(e)$\glossary{$\cO_q(e)$} is the quotient of the tensor algebra
$T(\Mat(e))$ by its two-sided quadratic ideal $I(e)$.
\end{definition}

\begin{definition} \label{bqca}The braided quiver coordinate algebra $\cO_q=\cO_q(\Mat_\dd(Q))$\glossary{$\cO_q,\cO_q(\Mat_\dd(Q))$} is the braided tensor product of algebras,
$$\cO_q(\Mat_\dd(Q)):= \overrightarrow{\bigotimes_{e\in E}} \cO_q(e).$$
\end{definition}

\begin{definition} The edge differential operator algebra $\cD_q(e)$\glossary{$\cD_q(e)$} is the quotient of the tensor algebra $T(\operatorname{Mat}(e)\oplus \operatorname{Mat}(e^\vee))$ by its two sided quadratic ideal $$\mathcal{I}:=I(e) + I(e^\vee) + I(e,e^\vee).$$
\end{definition}
\begin{definition}\label{bqa} The braided differential operator algebra $\cD_q(\Mat_\dd(Q))$\glossary{ $\cD_q(\Mat_\dd(Q))$} is the braided tensor product of algebras,
$$\cD_q(\Mat_\dd(Q)):= \overrightarrow{\bigotimes_{e\in E}} \cD_q(e).$$
\end{definition}

\begin{remark} Recall that for the tensor product of algebras in a braided tensor category the component subalgebras do not commute trivially; rather, they commute by the braiding:
$$\mu_{A\ot B}:=(\mu_A\ot\mu_B)\circ \sigma_{B,A}:A\ot B\ot A \ot B \to A\ot B.$$ 
Note, however, that edge algebras do commute trivially if they share no common vertex, since in this case they occupy distinct $\bt$-components of $\cC$.  However, we have an isomorphism $A\ot B\to B\ot A$ of $\cC$-algebras given by $\sigma_{B,A}^{-1}$; thus $\cO_q$ and $\cD_q$ are defined independently of the ordering of $v\in E$, up to isomorphism.
\end{remark}

\begin{remark}\label{teq1} The dependence on the parameter $t$\glossary{$t=1$ convention} appearing in the definition of $\cD_q(\Mat_d(Q))$ is inessential in the following sense:  for $t_1,t_2\neq 0$, the two algebras obtained by using $t_1$ or $t_2$ are isomorphic, by a simple rescaling of the generators (this phenomenon is common in the undeformed setting as well).  Thus, to ease notation, we set $t=1$, for the remainder of the paper.  The exception, however comes when we compute degenerations in Section \ref{sDAHA}, when we will work with formal power series in $\hbar$, and will set $t$ to $\hbar$, to compute the degeneration.\end{remark}

\section{Quantum groups and an RTT-type presentation for $\cO_q$ and $\cD_q$}\label{RTTpres}

In this section, we unfold the definitions of $\cO_q$ and $\cD_q$ in our primary examples of interest, namely those coming from $\cC_v=U_q(\mathfrak{gl}_{d_v})$-lfmod, so that $\cC=U_q(\g^\dd)$-lfmod.  We will see that for certain simple quivers $Q$, $\cO_q$ and $\cD_q$ are related to well-known constructions in the theory of quantum groups.
To begin, let us recall the quasi-triangular Hopf algebra $U_q(\mathfrak{gl}_N)$.  The discussion here has been adapted from \cite{JM}, \cite{KS}, where the relation to the Serre presentation is explained.

\subsection{The $R$-matrix on $\CC^N$}
We fix, for the remainder of this article, the following endomorphism $R$, of $\CC^N\ot\CC^N$:

\begin{equation}\label{eqn:R}\glossary{$R,R^{ij}_{kl}$}
R:= \tq\sum_{i}E_i^i\otimes E_i^i
+\sum_{i\neq j}E_i^i\otimes E_j^j+
(\tq-\tq^{-1})\sum_{i>j}E_i^j\otimes E_j^i.
\end{equation}
We note that $R$ satisfies the QYBE, and the Hecke condition from Section \ref{discussion}.  We define $R^{kl}_{ij},(R^{-1})^{kl}_{ij}\in\CC$, for $i,j,k,l=1,\ldots,N$ by:
\begin{eqnarray*}
R(e_{i}\otimes e_{j})
=\sum_{k,l}R_{ij}^{kl}(e_{k}\otimes e_{l}),\quad
R^{-1}(e_{i}\otimes e_{j})
=\sum_{k,l}(R^{-1})_{ij}^{kl}(e_{k}\otimes e_{l}).
\end{eqnarray*}
We have:
\begin{align*}
R^{ij}_{kl} &= q^{\delta^i_j}\delta^i_k\delta^j_l + (q-q^{-1})\theta(i-j)\delta^i_l\delta^j_k,\\
(R^{-1})^{ij}_{kl}&= q^{-\delta^i_j}\delta^i_k\delta^j_l - (q-q^{-1})\theta(i-j)\delta^i_l\delta^j_k,
\end{align*}
 where $\delta^i_j=1$ if $i=j, 0$ else, and $\theta(k)=1$ if $k>0, 0$ else.

\subsection{The Drinfeld-Jimbo quantum group $U_{q}(\mathfrak{gl}_N)$}\label{sec:QG}
Let $\widetilde{U}$ denote the free algebra with generators $l^{+i}_j$, and $l^{-k}_l$, where $i,j,k,l=1\ldots N$.  We organize the generators into matrices $L^+, L^-\in \Mat_N(\widetilde{U})\cong \widetilde{U}\ot \Mat_N(\CC)$:

$$L^+=\sum_{i,j} l^{+i}_j\ot E^j_i, \quad L^-=\sum_{k,l}l^{-k}_l\ot E^l_k.$$
For $M\in\Mat_N(\tilde{U})$, we define:
$$M_1:=M\ot\id,\,\, M_2=\id\ot M \in \Mat_N(\tilde{U})\ot\Mat_N(\tilde{U}).$$

\begin{definition} The Drinfeld-Jimbo quantum group $U_q(\mathfrak{gl}_N)$ is the quotient of $\widetilde{U}$ by the relations: 
\begin{align}{\label{l-rel1}}
&L_{1}^{\pm}L_{2}^{\pm}R=RL_{2}^{\pm}L_{1}^{\pm}, \quad
L_{1}^{-}L_{2}^{+}R=RL_{2}^{+}L_{1}^{-},\\
{\label{l-rel2}}
&l_{i}^{+i}l_{i}^{-i}=l_{i}^{-i}l_{i}^{+i}=1, \quad i=1, \ldots, N,\\
{\label{l-rel3}}
&l^{+i}_j=l^{-j}_i=0, \quad i>j.
\end{align}
$U$ is a Hopf algebra with the antipode $S$, coproduct $\Delta$ and counit $\epsilon$ given by:
\begin{align*}
S(L^{\pm})=(L^{\pm})^{-1},\quad \Delta(l_{j}^{\pm i})=\sum_{k}l_{k}^{\pm i}\otimes l_{j}^{\pm k},\quad
\text{ and }\quad \epsilon(l_{j}^{\pm i})=\delta_{j}^i.
\end{align*}
\end{definition}

\begin{remark}  Each of the relations in line \eqref{l-rel1} above is actually an $N^2\times N^2$ matrix of relations.  For instance equation \eqref{l-rel1} asserts, for all $i,j,m,n\in 1\cdots N$, the relations:
$$\sum_{k,l} l^{+i}_kl^{+j}_lR^{kl}_{mn} = \sum_{o,p} R^{ij}_{op}l^{+p}_ml^{+o}_n.$$
We shall use such notation frequently in this and future sections without further comment.
\end{remark}

\begin{definition} The vector representation $\rho: U \to \operatorname{End}(\mathbb{C}^{N})$ is defined on generators by:
\begin{align*}
\rho_{V}(l_{j}^{+i})=\sum_{\alpha,\beta}R_{\beta j}^{\alpha i}E_\alpha^\beta, \quad
\rho_{V}(l_{j}^{-i})=\sum_{\alpha,\beta}(R^{-1})^{i\alpha}_{j\beta}E_\alpha^\beta.
\end{align*}
\end{definition}

\subsection{The locally finite part $U'$ of $U$}
The Hopf algebra $U$ acts on itself via the adjoint action: \glossary{$\rhd$}
$$x \rhd y := x_{(1)}yS(x_{(2)}).$$
\begin{definition}\emph{The locally finite subalgebra} $U'$ is the subalgebra of $U$ of vectors which generate a finite dimensional vector space under the adjoint action.\end{definition}
We will make use of the following explicit description of $U'$.  We define $\widetilde{l}^i_j\in U$ by $\widetilde{l}^i_{j} := \sum_k l^{+i}_{k}S(l^{-k}_{j})$.  We define $\widetilde{L}:=L^+S(L^-)$, so that \mbox{$\widetilde{L} = \sum_{i,j} \widetilde{l}^i_jE^j_i$}.\glossary{$\widetilde{L}$, $\widetilde{l}^i_j$}

\begin{theorem}(see \cite{KS})\label{lfdesc}
\begin{enumerate}
\item $U'$ is generated by the $\widetilde{l}^i_{j}$, and the inverse $\operatorname{det}_q^{-1}=l^{-1}_{1}\ldots l^{-N}_{N}$ of the $q$-determinant.
\item $U'$ is a left co-ideal: we have $\Delta(U')\subset U\ot U'$.  The coproduct on $U'$ is given by:
$$\Delta(\widetilde{l}^i_{j})=\sum_{s,t} l^{+i}_{s}S(l^{-t}_{k})\otimes \widetilde{l}^s_{t},\,\,\,\Delta(\operatorname{det}^{-1}_q)=\operatorname{det}^{-1}_q\otimes \operatorname{det}^{-1}_q$$
\end{enumerate}
\end{theorem}

Let $U^+$ denote the subalgebra generated by the $\widetilde{l}^i_{j}$.  Item (2) above implies that $U^+$ is a co-ideal subalgebra in $U'$.

\subsection{Braided quiver coordinate algebra}\label{RTTsec1}

Fix a quiver $Q=(V,E)$, and a dimension vector $\dd:V\to\ZZ_{\geq 0}$. We specialize\glossary{$\cC_v$,$W_v$} $\cC_v=U_q(\mathfrak{gl}_{d_v})$-mod, with $W_v=\CC^{d_v},$ its defining representation. In this case, the matrix representation $(R^v)^{ij}_{kl}$ of the universal $R$-matrix is defined relative to the standard basis of $\CC^{d_v}$, so that $\sigma_{W_v,W_v}(w_i\ot w_j)=(R^v)^{kl}_{ij}w_l\ot w_k$.

Recall that the identities $(S\ot \id)(R)=R^{-1}$ and $(S\ot S)(R)=R$ imply the formulas:
\begin{align*}\sigma_{W_v^*,W_v}(v^i\ot v_j)&= \sum_{\alpha,\beta}(R^{-1})^{i\alpha}_{\beta j} v_\alpha\ot v^\beta,\\
\sigma_{W_v,W_v^*}^{-1} (v^i\ot v_j)&=\sum_{\alpha,\beta}R^{i\alpha}_{\beta j}v_\alpha\ot v^\beta,\\
\sigma_{W_v^*,W_v^*}(v^i\ot v^j) &= \sum_{\alpha,\beta}R^{ij}_{\alpha\beta} v^\beta\ot v^\alpha.\\
\end{align*}
%
In the definitions to follow, we introduce the following four matrices (where the $a(e)^i_j$ are formal symbols):\glossary{$R_{21}$, $A^e_1$, $A^e_2$}
 
 \begin{align}\label{RTTnot}R:= \sum_{i,j} R^{ij}_{kl}(E^k_i\ot E^l_j),\quad R_{21}:= \sum_{i,j} R^{ji}_{lk}(E^k_i\ot E^l_j),\\ A_1^e:=\sum_{ij} a(e)^i_j(E^j_i\ot \id), \quad A_2^e:=\sum_{ij} a(e)^i_j(\id\ot E^j_i).\nonumber\end{align}
 
 The following is a more concrete reformulation of Definition \ref{bqca}:

\begin{definition} \label{RTTbqca}The braided quiver coordinate algebra, $\cO_q(\Mat_\dd(Q))$, \glossary{$\cO_q(\Mat_\dd(Q))$} is the algebra generated by elements $a(e)^i_j$, for $e\in E$, $i=1,\ldots d_{\alpha(e)}$, and $j=1,\ldots,d_{\beta(e)}$, subject to:
\begin{enumerate}
\item Relations between generators on the same edge:
\begin{align*}\overset{v}{\bullet}\overset{e}{\rightarrow}\overset{w}{\bullet}:&&
R^vA^e_2A^e_1&=A^e_1A^e_2R_{21}^w,\\
\overset{v}{\bullet}\overset{e}{\rightloop}:&& R^v_{21}A^e_1R^vA^e_2 &= A^e_2R_{21}^vA^e_1R^v,\end{align*}

\item Relations between generators on distinct edges (assume $e<f$):
\begin{align*}
{\bullet}\overset{f}{\rightarrow}\bullet\quad\bullet\overset{e}{\rightarrow}{\bullet}:&& A^f_1A^e_2 &= A^e_2A^f_1\\
\overset{v}{\bullet}\underset{f}{\overset{e}{\rightrightarrows}}\overset{w}{\bullet}:&& A^f_1A^e_2 &= R^v A^e_2A^f_1R^w\\
\overset{v}{\bullet}\underset{f}{\overset{e}{\rightleftarrows}}\overset{w}{\bullet}:&&
A^f_1R^vA^e_2&=A^e_2(R^w)^{-1}A^f_1\\
{\bullet}\overset{e}{\rightarrow}\overset{v}{\bullet}\overset{f}{\rightarrow}\bullet:&&
A^f_1A^e_2&= A^e_2(R^v)^{-1}A^f_1,\\
\bullet\overset{e}{\leftarrow}\overset{v}{\bullet}\overset{f}{\leftarrow}\bullet:&&
A^f_1R^vA^e_2&=A^e_2A^f_1\\
{\bullet}\overset{e}{\rightarrow}\overset{v}{\bullet}\overset{f}{\leftarrow}\bullet:&&
A^f_1A^e_2&=A^e_2A^f_1R^v\\
 {\bullet}\overset{e}{\leftarrow}\overset{v}{\bullet}\overset{f}{\rightarrow}\bullet:&&A^f_1A^e_2&=R^vA^e_2A^f_1,\\
 \bullet \overset{e}{\rightarrow}\overset{v}{\bullet} \overset{f}{\rightloop}:&&A^f_1A^e_2&=A^e_2(R^v)^{-1}A^f_1R^v\\
 \bullet \overset{e}{\leftarrow}\overset{v}{\bullet} \overset{f}{\rightloop}:&&A^f_1R^vA^e_2&=R^vA^e_2A^f_1\\
 \bullet \overset{f}{\rightarrow}\overset{v}{\bullet} \overset{e}{\rightloop}:&&A^f_1R^vA^e_2&=A^e_2A^f_1R^v\\
 \bullet \overset{f}{\leftarrow}\overset{v}{\bullet} \overset{e}{\rightloop}:&&A^f_1A^e_2&=R^vA^e_2(R^v)^{-1}A^f_1\\
 \overset{e}{\leftloop}\overset{v}{\bullet} \overset{f}{\rightloop}:&&A^f_1R^vA^e_2(R^v)^{-1}&=R^vA^e_2(R^v)^{-1}A^f_1\\
%
\end{align*}

\end{enumerate}
\end{definition}

\subsection{Braided quiver differential operator algebra}\label{RTTsec2}
To the notation of equation \eqref{RTTnot}, we add: \glossary{$D^e_1$, $D^e_2$, $\Omega$}
$$D^e_1:= \sum_{k,l} \partial(e)^k_l(E^l_k\ot \id),\quad D^e_2:= \sum_{k,l} \partial(e)^k_l(\id\ot E^l_k),\quad \Omega:=\sum_{i,j} E^i_j\ot E^j_i.$$

\begin{definition} The braided quiver differential operator algebra, $\cD_q(\Mat_\dd(Q))$, \glossary{$\cD_q(\Mat_\dd(Q))$}is the algebra generated by elements $a(e)^i_j$ and $\partial(e)^k_l$, for $e\in E,$ with  $i,l=1,\ldots,d_{\alpha(e)},$ and $j,k=1,\ldots, d_\beta(e)$, subject to:\label{Dqdefn}
\begin{enumerate}
\item The generators $a(e)^i_j$ satisfy the same relations amongst themselves as $a(e)^i_j$ in (1) and (2) of Definition \ref{RTTbqca}.
\item The generators $\partial(e)^k_l$ satisfy the same relations amongst themselves as $a(e^\vee)^k_l$ in (1) and (2) of Definition \ref{RTTbqca}.
\item For $e\neq f$, the generators $a(e)^i_j$ and $\partial(e)^k_l$ satisfy the same cross relations as $a(e)^i_j$ and $a(e^\vee)^k_l$, respectively in (2) of Definition \ref{RTTbqca}.
\item For generators $a(e)^i_j$ and $\partial(e)^k_l$ on the same edge, we have the cross relations:
\begin{align*}\overset{v}{\bullet}\overset{e}{\rightarrow}\overset{w}{\bullet}:&&
D^e_2(R^v)^{-1}A^e_1&=A^e_1R^wD^e_2 + \Omega,\\
\overset{v}{\bullet}\overset{e}{\rightloop}:&& R_{21}^vD^e_1R^vA^e_2 &= A^e_2R_{21}^vD^e_1(R_{21}^v)^{-1},\end{align*}

\end{enumerate}
\end{definition}

\subsection{Familiar examples}

Definition \ref{bqca} encompasses many standard examples in the theory of quantum groups when applied to small quivers; these are illustrated below.  To simplify notation, we do not specify ranges of free indices in equations, when the range is clear from context.
\begin{example}\label{Kron} The Kronecker quiver.  Let $Q=\overset{\alpha}{\bullet}\xrightarrow{e}\overset{\beta}{\bullet}$.  Choose dimensions $d_\alpha$, $d_\beta$ and let $\cC_\alpha=U_q(\mathfrak{gl}_{d_\alpha})$, and $\cC_\beta=U_q(\mathfrak{gl}_{d_\beta})$. Let $V_\alpha=\CC^{d_\alpha}$, $V_\beta=\CC^{d_\beta}$.  Then $\cO_q(\Mat_\dd(Q))$ is the quotient of the free algebra with generators $a^i_j$, with $i=1,\ldots, d_\alpha,$ and $j=1,\ldots, d_\beta$, by the relations:
$$\sum_{k,l}R^{ji}_{kl}a^l_ma^k_n=\sum_{k,l}a^j_l a^i_k R^{kl}_{mn},$$
or equivalently:
\begin{align}\label{eFRT-relns} 
a^i_ma^j_n  &= q^{\delta_{mn}}a^j_na^i_m + \theta(m-n)(q-q^{-1})a^j_ma^i_n, &(i>j),\\
a^i_ma^i_n &= q^{-1}a^i_na^i_m, &(m>n).\nonumber
\end{align}
Similarly, $\cD_q(\Mat_\dd(Q))$ is the quotient of the free algebra with generators $a^i_j,\partial^k_l$ with $i, l=1,\ldots,d_\alpha,$ and $j,k=1,\ldots,d_\beta$, by the relations:
$$
\sum_{k,l}R^{ij}_{kl}a^l_ma^k_n=\sum_{k,l}a^i_la^j_kR^{kl}_{mn}, \quad \sum_{k,l}R^{ij}_{kl}\partial^l_m\partial^k_n=\sum_{k,l}\partial^i_l\partial^j_kR^{kl}_{mn}$$
$$\sum_{k,l}\partial^i_k (R^{-1})^{jk}_{lm}a^l_n = \sum_{k,l}a^j_k R^{ki}_{nl}\partial^l_m +q\delta^i_n\delta^j_m,$$

Equivalently, the generators $a^i_j$ and $\partial^k_l$ satisfy relations \eqref{eFRT-relns} amongst themselves, and the cross relations:
\begin{align*}
\partial^i_ma^j_n &= q^{\delta_{in}+\delta_{jm}}a^j_n\partial^i_m + \delta_{in}q^{\delta_{jm}}(q-q^{-1})\sum_{p>i}a^j_p\partial^p_m + \delta_{jm}(q^2-1)\sum_{p>j}\partial^i_pa^p_n + \delta_{in}\delta_{jm}
\end{align*}
We observe that $\cO_q$ is the \emph{equivariant FRT algebra} (see Proposition \ref{Oeflat}), while $\cD_q$ is, apparently, new.
\end{example}

\begin{example} The quantum plane.  Let $Q$ be the Kronecker quiver with $d_\alpha=1$, and $d_\beta=N\in\NN$.  The defining representation for $U_q(\mathfrak{gl}_1)$ has $R_{V,V}=q\in\CC^\times$, so that setting $x_j:=a^1_j$, we have that $\cO_q(\Mat_\dd(Q))$ is a quotient of the free algebra generated by $x_1,\ldots x_N$ by the relations:
\begin{align}x_ix_j&=q^{-1}x_{j}x_{i},& (i>j).\label{qplane}\end{align}
Likewise, $\cD_q(\Mat_\dd(Q))$ is the quotient of the free algebra generated by $x_j,\partial^k$, with $j,k=1,\ldots, N$, by the relations \eqref{qplane}, and: 
\begin{align*}
\partial^i\partial^j&=q\partial^j\partial^{i},&(i>j)\\
\partial^ix_j &= qx_j\partial^i&(i\neq j).\\
\end{align*}
\vspace{-.5in}
$$\partial^ix_i = q^2x_i\partial^i + (q^2-1)\sum_{k>i}x_k\partial^k +q.$$

In this case, the relations essentially reduce to the relations for the ``quantum Weyl algebra" from \cite{GZ}.
\end{example}

\begin{example}\label{JordanQuiver} The Jordan normal form quiver.  Let $Q$ have a single vertex $v$, and loop $e:v\to v$.  Let $\cC=U_q(\mathfrak{gl}_N)$-mod, and $V=\CC^N$.  Then 
$\cO_q(\Mat_\dd(Q))$ is the quotient of the free algebra with generators $a^i_j$, for $i,j=1,\ldots N$, with relations:
$$ \sum_{k,l,o,p}R^{ji}_{kl}a^l_pR^{pk}_{mo}a^o_n = \sum_{p,q,s,t}a^j_pR^{pi}_{sq}a^q_tR^{ts}_{mn},$$
or equivalently:
\begin{align}\label{OqJord}
\\
a^i_ma^j_n &= q^{\delta_{in}+\delta_{mn}-\delta_{mj}}a^j_na^i_m + q^{\delta_{im}-\delta_{jm}}(q-q^{-1})\theta(n-m)a^j_ma^i_n \quad\quad\quad\quad\quad(i>j)\nonumber\\
&\phantom{===} + q^{\delta_{mn}-\delta_{jm}}\delta_{in}(q-q^{-1})\sum_{p>i}a^j_pa^p_m - \delta_{mj}(1-q^{-2})\sum_{p>j}a^i_pa^p_n\nonumber\\
&\phantom{===} + q^{-\delta_{jm}}\delta_{im}(q-q^{-1})^2\theta(n-m)\sum_{p>i}a^j_pa^p_n.\nonumber\\
a^i_ma^i_n &= q^{\delta_{in}-\delta_{im}-1}a^j_na^i_m + q^{-1-\delta_{jm}}\delta_{in}(q-q^{-1})\sum_{p>i}a^i_pa^p_m \quad\quad\quad\quad\quad\quad\quad (m>n)\nonumber
\end{align}

Likewise, $\cD_q(\Mat_\dd(Q))$ is the quotient of the free algebra with generators $a^i_j,\partial^k_l$, for $i,j,k,l=1,\ldots, N$, and relations:
$$ \sum_{k,l,o,p}R^{ji}_{kl}a^l_pR^{pk}_{mo}a^o_n = \sum_{p,q,s,t}a^j_pR^{pi}_{sq}a^q_tR^{ts}_{mn},$$
$$ \sum_{k,l,o,p}R^{ji}_{kl}\partial^l_pR^{pk}_{mo}\partial^o_n = \sum_{p,q,s,t}\partial^j_pR^{pi}_{sq}\partial^q_tR^{ts}_{mn},$$
$$\sum_{k,l,o,p}R^{ji}_{kl}\partial^l_pR^{pk}_{mo}a^o_n = \sum_{k,l,o,p}a^j_lR^{li}_{pk}\partial^k_o(R^{-1})^{po}_{nm}.$$
Equivalently, the $a^i_j$ and $\partial^i_j$ satisfy the relations \eqref{OqJord} amongst themselves, and the cross relations (for $i, j, m, n$ distinct):
\begin{align}\label{DqJord}\\
\partial^i_ma^j_n &= q^{\delta_{i,n}-\delta_{m,n}-\delta_{m,j}-\delta_{i,j}}a^j_n\partial^i_m - \theta(j-i)q^{\delta_{i,m}-\delta_{m,j} -\delta_{i,j}}(q-q^{-1})\partial^j_ma^i_n\nonumber\\
&\phantom{===} - \theta(m-n)(q-q^{-1})q^{\delta_{i,m}-\delta_{m,j}-\delta_{i,j}}a^j_m\partial^i_n
 -(q-q^{-1})q^{\delta_{m,j}}\delta_{m,j}\sum_{p>k}\partial^i_pa^p_n\nonumber\\
&\phantom{===} - (q-q^{-1})^2q^{-\delta_{m,j}-\delta_{i,j}}\delta_{i,m}\theta(j-i)\sum_{p>k}\partial^j_pa^p_n\nonumber\\
&\phantom{===} - (q-q^{-1})q^{-\delta_{m,n}-\delta_{m,j}-\delta_{i,j}}\delta_{i,n}\sum_{p>i}a^j_p\partial^p_m\nonumber\\
&\phantom{===} + (q-q^{-1})^2q^{-\delta_{m,j}-\delta_{i,j}}\theta(m-n)\delta_{i,m}\sum_{p>i}a^j_p\partial^p_n\nonumber
\end{align}

In this case, $\cO_{q}$ is the well-known reflection equation algebra \cite{Do}, \cite{Ma}, \cite{Mu}, while $\cD_{q}$ is the algebra $\mathbb{D}^+$ of polynomial quantum differential operators on quantum $GL_n$, as studied in \cite{VV}.
\end{example}

\subsection{New examples}
New examples of interest are detailed below.  For two $\CC$-algebras $A, B$, we let $A\ast B$ denote their free product, and we use the notation $\displaystyle{\overset{\ast}{\prod_{i\in I}}A_i}$ for iterated free products.

\begin{example} The Calogero-Moser quiver\glossary{Calogero-Moser quiver}.  Let $Q$ and $d$ be the Calogero-Moser quiver and dimension vector, $(Q,d)= \overset{1}{\bullet}\rightarrow\overset{n}{\bullet}\rightloop$.  Then $\cO_q(\Mat_\dd(Q))$ is the quotient of the free product, $\cO_q(\overset{1}{\bullet}\rightarrow\overset{n}{\bullet}) \ast \cO_q(\overset{n}{\bullet}\rightloop)$, by the relations:
$$\sum_{k,l} x_k R^{k i}_{j l}a^l_m = \sum_{k,l} a^i_k x_l R^{l k}_{jm}.$$
Likewise, $\cD_q(\Mat_\dd(Q))$ is the quotient of the free product, $\cD_q(\overset{1}{\bullet}\rightarrow\overset{n}{\bullet}) \ast \cD_q(\overset{n}{\bullet}\rightloop)$, by the relations:
$$\sum_{k,l}x_k R^{k i}_{j l}a^l_m = \sum_{k,l}a^i_k x_l R^{l k}_{jm}, \quad
\sum_{k,l}x_k R^{k i}_{j l}\partial^l_m = \sum_{k,l}\partial^i_k x_l R^{l k}_{jm}$$
$$\partial^ia^j_m = \sum_{k,l,o,p}R^{ij}_{kl}a^l_o(R^{-1})^{k o}_{p m} \partial^p, \quad
\partial^i\partial^j_m = \sum_{k,l,o,p}R^{ij}_{kl}\partial^l_o(R^{-1})^{k o}_{p m} \partial^p$$
\end{example}

\begin{example}\label{starexample} Star shaped quiver. \glossary{star-shaped quiver} Let $Q$ be the star-shaped quiver, with legs of length $l_1,\ldots,l_m$, and with nodal vertex $v_0$.  We adopt the following labelling convention on $Q$.  The vertex set of $Q$ is:
$$V:=\{v_{\alpha\beta}\,\,|\,\,\alpha=1, \ldots, m,\,\beta=0, \ldots l_i\},$$where each $v_{\alpha\beta}$ is on the $\alpha$th leg, at a distance of $\beta$ edges from the node, and $v_0=v_{\alpha,0}$, for all $\alpha=1\,\ldots, m$.  The edge set of $Q$ is:
$$E:=\{e_{\alpha,\beta}:v_{\alpha,\beta+1}\to v_{\alpha,\beta}\,\,|\,\,\alpha=1,\ldots m, \beta=0,\ldots, l_\alpha-1\}.$$
The labelling is depicted below:

$$
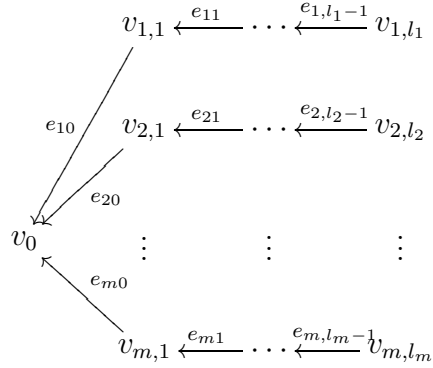
\xymatrix{ & \ar@{->}[ddl]_{e_{10}} v_{1,1} & \ar@{->}[l]_{e_{11}} \cdots & \ar@{->}[l]_{e_{1,l_1-1}} v_{1,l_1}\\
& \ar@{->}[dl]^{e_{20}} v_{2,1} & \ar@{->}[l]_{e_{21}} \cdots & \ar@{->}[l]_{e_{2,l_2-1}} v_{2,l_2}\\
v_0 & \vdots  & \vdots & \vdots \\
& \ar@{->}[ul]_{e_{m0}} v_{m,1} & \ar@{->}[l]_{e_{m1}} \cdots & \ar@{->}[l]_{e_{m,l_m-1}} v_{m,l_m}}
$$

We choose for the ordering on the edges the natural lexicographic ordering on the indices.  We set $d_v=1$ for all $v\neq v_0$, and $d_{v_0}=n$; we call such $\dd$ the Calogero-moser dimension vector for $Q$.  By Example \ref{Kron}, for $\alpha=1,\ldots m$, $\beta=1\ldots l_i-1$, each edge algebra $\cD_q(e_{ij})$ has two generators, which we denote $x_{\alpha}$ and $\partial_{\beta}$. Likewise, each edge algebra $\cD_q(e_{i,0})$ has $2n$ generators, which we denote $y_{\alpha1}, \ldots, y_{\alpha n}$, $\xi_\alpha^1, \ldots, \xi_\alpha^n$. Then, $\cD_q(\Mat_\dd(Q))$ is the quotient of the free product, $\displaystyle{\overset{\ast}{\prod_{e_{\alpha\beta}\in E}}\cD_q({\bullet}\xrightarrow{e_{\alpha\beta}}\bullet)},$ by the relations that all generators without a common vertex commute, and cross-relations on the remaining edges:
$$
x_{\alpha,\beta-1}x_{\alpha\beta}=qx_{\alpha\beta}x_{\alpha,\beta-1},\quad \partial_{\alpha,\beta-1}\partial_{\alpha\beta}=q^{-1}\partial_{\alpha\beta}\partial_{\alpha,\beta-1}$$
$$\partial_{\alpha\beta}x_{\alpha,\beta-1}=qx_{\alpha,\beta-1}\partial_{\alpha\beta},\quad
x_{\alpha\beta}\partial_{\alpha,\beta-1}=q^{-1}\partial_{\alpha,\beta-1}x_{\alpha\beta},$$
$$y_{\beta i}y_{\alpha j} = \sum_{k,l}y_{\alpha k}y_{\beta l}R^{lk}_{ij},\quad\xi^i_\beta \xi^j_\alpha = \sum_{k,l} R^{ij}_{kl}\xi^l_\beta \xi^k_\beta, \quad (\textrm{for } \alpha<\beta),$$
$$\xi^i_\alpha y_{\beta j}=\sum_{k,l}y_{\beta k}R^{ki}_{jl}\xi_\alpha^l,\quad \xi_\beta^iy_{\alpha j} = \sum_{k,l} y_{\alpha k} (R^{-1})^{ik}_{lj}\xi_\beta^l.
$$
\edit{and a few more for the nodal vertex}
\end{example}

\begin{remark}
It has been suggested to us by B. Webster that the case when $Q$ is arbitrary non-Dynkin, but $d_v=1$ for all $v$ should yield quantizations of hypertoric varieties associated to $Q$.  We hope to study such examples in the future.
\end{remark}

\subsection{Monomial notation}
In order to denote monomials in the generators of $\cO_q$ and $\cD_q$, we introduce the following shorthand. Let $I$ be an ordered list of triples $I=(e_i\in E, m_i\in\{1,\ldots d_{\alpha(e)}\},n_i\in\{1,\ldots,d_\beta(e)\})_i$, and $J$ an ordered list of triples $J=(f_i\in E^\vee,o_i\in \{1,\ldots d_{\beta(e)}\},p_i\in\{1,\ldots,d_\alpha(e)\})_i$, we denote the products\glossary{$a_I$,$\partial_J$, standard monomials}
\begin{align*}
a_I&:=a(e_{1})^{m_1}_{n_1}\cdots a(e_k)^{m_k}_{n_k}\\
\partial_J&:=\partial(f_1)^{o_1}_{p_1}\cdots \partial(f_l)^{o_l}_{p_l}.
\end{align*}
When there is no risk of confusion, we omit the specification of the edge in the notation (e.g, we write $a^i_j$ instead of $a(e)^i_j$). The list $I$ will be said to be ordered, if for all $i<j$, either $e_i<e_j$. or $e_i=e_j$ and $m_i<m_j$, or $e_i=e_j,m_i=m_j$ and $n_i\leq n_j$.  Likewise the list $J$ will be said to be ordered, if for all $i<j$, either  $f_i<f_j$. or $f_i=f_j$ and $o_i<o_j$, or $f_i=f_j,o_i=o_j$ and $p_i\leq p_j$.  Monomials $a_I\partial_J$, for $I,J$ ordered, will be called \emph{standard monomials}.
\section{Flatness of the algebras $\cO_q$ and $\cD_q$}\label{flatness-section}
In the present section, we prove that the algebras $\cO_q$ and $\cD_q$ constructed in previous sections are flat noncommutative deformations of their classical counterparts, the algebras $\cO(\operatorname{Mat}_\dd(Q))$ and $\cD(\operatorname{Mat}_\dd(Q))$.
More precisely, we show that the set of standard monomials form a basis of $\cO_q$ and $\cD_q$.

\begin{proposition}\label{Oeflat}  We have the following descriptions for $\cO_q(e)$:
\begin{enumerate}
\item If $\alpha(e)\neq \beta(e)$, then $O_q(e)$ is twist equivalent to the FRT algebra via the tensor equivalence $\sigma\bt\id:\cC\bt\cC\to\cC^{\vee}\bt\cC$.
\item If $\alpha(e)=\beta(e)$, then $O_q(e)$ is isomorphic to the reflection equation algebra.
\end{enumerate}
\end{proposition}
\begin{proof}The $\cC^{\vee}\bt\cC$-algebra $\cO'_q(e)$, twist-equivalent to $\cO_q(e)$, has the same underlying vector space as $\cO_q(e)$, with multiplication given by $m':=m\circ (R^{-1}\bt\id),$ where $m$ denotes the product in $\cO_q(e)$.  In particular, $\cO'_q(e)$ is generated by elements $\tilde{a}^i_j$, $i=1,\ldots, n, j=1,\ldots, m$, with relations:
$$R^{ij}_{op}\tilde{a}^o_m\tilde{a}^p_n=a^i_ma^j_n=(R^{ji}_{op})^{-1}a^o_ka^p_lR^{lk}_{mn}=\tilde{a}^j_k\tilde{a}^i_lR^{lk}_{mn},$$ which are the relations of the FRT algebra.  On the other hand if $\alpha=\beta$, we have seen in Example \ref{JordanQuiver} that we recover the relations of the reflection equation algebra.
\end{proof}

$\cO_q$ is defined as a braided tensor product of the edge algebras $\cO_q(e)$, which are flat by Proposition \ref{Oeflat}, together with the well-known flatness of the FRT and RE algebras (see, e.g. \cite{KS}).  More precisely, we have:
\begin{corollary} The algebra $\cO_q$ is a flat deformation of the algebra $\cO(\operatorname{Mat}_\dd(Q))$. A basis of $\cO_q$ is given by the set of standard monomials $a_I$.
\end{corollary}

In fact, the analogous statement holds for $\cD_q$, as well.  The proof is modeled on the proof of Theorem 1.5 of \cite{GZ}, which is a special case.  We have:

\begin{theorem}  The algebra $\cD_q$ is a flat deformation of the algebra $\cD(\operatorname{Mat}_\dd(Q))$. A basis of $\cD_q$ is given by the set of standard monomials $a_I$, $\partial_J$. \label{flatness}
\end{theorem}
\begin{proof}
Since we have defined $\cD_q$ as a braided tensor product of its edge algebras, we need only to prove flatness for each edge algebra $\cD_q(e)$.  By Theorem \ref{Oeflat}, it suffices to prove that $\cD_q(e)\cong \cO_q(e)\ot \cO_q(e^\vee)$, as a vector space.  We make use of the following lemma, which generalizes \cite{GZ}, Lemma 1.6.

\begin{lemma} In the tensor algebra $T(\Mat(e)\oplus \Mat(e^\vee))$, we have the following containments:
\begin{enumerate}
\item $T(\Mat(e^\vee)) I(e) \subset I(e) T(\Mat(e^\vee)) + I(e,e^\vee)$.
\item $I(e^\vee)T(\Mat(e))  \subset T(\Mat(e))I(e^\vee) +I(e,e^\vee)$.
\end{enumerate}
\end{lemma}
\begin{proof} We prove (1) by direct computation; (2) then follows by a similar proof, due to the symmetry in the definition of $I(e,e^\vee)$.  For the first claim, it suffices to show that, for all $o,p,i,j,m,n$, we have
$$\partial^o_p(R^{ij}_{kl}a^l_ma^k_n-a^i_la^j_kR^{kl}_{mn})\in I(e)T(e^\vee) + I(e,e^\vee).$$  This is equivalent to showing that $A^{suo}_{nmv}\in I(e)T(e^\vee) + I(e,e^\vee)$, for all $u,s,o,v,m,n$, where:
$$A^{suo}_{nmv}:=(R^{-1})^{ut}_{jv}(R^{-1})^{sp}_{it}\partial^o_p(R^{ij}_{kl}a^l_ma^k_n-a^i_la^j_kR^{kl}_{mn}),$$
as these differ by an invertible linear transformation, and so generate the same subspace.  We let
\begin{align*}A & := \sum_{u,s,o,v,m,n} A^{suo}_{nmv}(E^n_s \ot E^m_u \ot E^v_o)\\
&= D_3R^{-1}_{13}R^{-1}_{23}R_{12}A_2A_1 - D_3R^{-1}_{13}R^{-1}_{23}A_1A_2R_{21},
\end{align*}
in the notation of Section \ref{sec:QG},
so that the matrix coefficients of $A$ are precisely the $A^{suo}_{nmv}$.  We compute:
\begin{align*}
 A &= D_3\underbrace{R^{-1}_{13}R^{-1}_{23}R_{12}}_{QYBE}A_2A_1 - \underbrace{D_3R^{-1}_{13}A_1}_{I(e,e^\vee)}R^{-1}_{23}A_2R_{21}\\
   &= R_{12}\underbrace{D_3R^{-1}_{23}A_2}_{I(e,e^\vee)}R^{-1}_{13}A_1 - A_1R_{13}\underbrace{D_3R^{-1}_{23}A_2}_{I(e,e^\vee)}R_{21} - \Omega_{13} R^{-1}_{23}A_2R_{21}\\
   &= R_{12}A_2R_{23}\underbrace{D_3R^{-1}_{13}A_1}_{I(e,e^\vee)}  +  R_{12}\Omega_{23} R^{-1}_{13}A_1 - A_1A_2\underbrace{R_{13}R_{23}R_{21}}_{QYBE}D_3\\
   &\phantom{===} - A_1R_{13}\Omega_{23}R_{21} - \Omega_{13}R^{-1}_{23}A_2R_{21}\\
   &= R_{12}A_2A_1R_{23}R_{13}D_3 + R_{12}A_2R_{23}\Omega_{13}  + \underbrace{R_{12}\Omega_{23} R^{-1}_{13}}_{\textrm{cancel inv.}}A_1 - A_1A_2R_{21}R_{23}R_{13}D_3\\
   &\phantom{===} - A_1R_{13}\Omega_{23}R_{21} - \Omega_{13}R^{-1}_{23}A_2R_{21}\\
   &= (R_{12}A_2A_1 - A_1A_2R_{21})R_{23}R_{13}D_3 + \underbrace{(R_{12}-R_{21}^{-1})}_{\textrm{Hecke reln.}}A_2R_{23}\Omega_{13}  + \Omega_{23}A_1\\
   &\phantom{===} - A_1R_{13}\Omega_{23}R_{21}\\
   &= (R_{12}A_2A_1 - A_1A_2R_{21})R_{23}R_{13}D_3 + (q-q^{-1})\Omega_{12}A_2R_{23}\Omega_{13}  + \Omega_{23}A_1\\
   &\phantom{===} - A_1R_{13}\Omega_{23}R_{21}\\
   &= (R_{12}A_2A_1 - A_1A_2R_{21})R_{23}R_{13}D_3\\
&\phantom{===} A_1\Omega_{23}\underbrace{((q-q^{-1})R_{12}\Omega_{12}  + 1
    - R_{12}R_{21})}_{\textrm{Hecke reln.}}\\
   &= (R_{12}A_2A_1 - A_1A_2R_{21})R_{23}R_{13}D_3.\\
\end{align*}
Comparing matrix coefficients, the above reads:
$$A^{suo}_{nmv} =R^{kl}_{nj}R^{io}_{ml}(R^{su}_{pa}a^a_ia^p_k - a^s_pa^u_t R^{tp}_{ik}) \partial^j_v \subset I(e)O_{e^\vee},$$
as claimed.
\end{proof}

To finish the proof of the theorem, we first observe that \mbox{$\cD_q(e)\cong S/(I(e) + I(e^\vee))$}, where \mbox{$S=T(\Mat(e)\oplus\Mat(e^\vee))/I(e,e^\vee)$}.  Every element of $S$ can be uniquely reduced to a sum $\sum C_{IJ} a_I\partial_J$, where $C_{IJ}\in\CC$, by relations $I(e,e^\vee)$, by straightforward application of the diamond lemma. Thus the multiplication $m:T(\Mat(e))\ot T(\Mat(e^\vee)) \to S$ is a linear isomorphism.
By the lemma, the two-sided ideal $\langle I(e)\ot 1 + 1\ot I(e^\vee)\rangle\subset S$ lies in the image under $m$ of the linear subspace $I(e)\ot T(\Mat(e^\vee))+ T(\Mat(e))\ot I(e^\vee)$, which implies that
$$m:T(\Mat(e))/I(e)\ot T(\Mat(e^\vee))/I(e^\vee)\to \cD_q(e)$$
is a linear isomorphism, as desired.
\end{proof}

\begin{corollary} The identification $\cO_q\cong \cD_q/(\sum_{e\in E}\Mat(e^\vee))\cD_q(e)$ as objects of $\cC$ makes $\cO_q(\Mat_\dd(Q))$ a $\cD_q(\Mat_\dd(Q))$-module in $\cC$, $q$-deforming the usual $\mathbb{G}^\dd$-equivariant action of $\cD(\Mat_\dd(Q))$ on $\cO(\Mat_\dd(Q))$.\end{corollary}

\section{Independence of $\cD_q(\Mat_\dd(Q))$ on the orientation of $Q$}\label{q-Fourier}\glossary{$\mathcal{F}$}
The algebra of differential operators on a finite dimensional vector space, $V=\langle e_1,\ldots,e_n\rangle$ with dual basis $V^*=\langle f_1,\ldots, f_n\rangle$, has a Fourier transform automorphism $\mathcal{F}$, induced by the symplectomorphism on the symplectic vector space $V\oplus V^*$, $e_i \mapsto f_i$, $f_i\mapsto -e_i$.  In this section we show that a certain localization $\cD_q(e)^\circ$ of the edge differential operator algebras $\cD_q(e)$ admit analogous isomorphisms, which may be extended to (a localization $\cD_q(\Mat_d(Q))^\circ$ of) $\cD_q(\Mat_d(Q))$, by the identity on the other $\cD_q(f)$ subalgebras.  In particular, this implies that the algebra $\cD_q(\Mat_\dd(Q))$ does not depend on the orientation of $Q$, up to isomorphism.  The results of this section should also be compared to Section 2 of \cite{C-BS}, of which they are a quantization.

\subsection{Braided Fourier transform on $\cD_q(e)$ when $e$ is not a loop}
\subsubsection{Easy case: } $e=\overset{1}{\bullet}\to \overset{1}{\bullet}$.
We work this example out for the sake of clarity, before considering the general situation.  In this case, we have:

$$\cD_q(e) = \CC\langle\partial,a\rangle \Big/ \langle \partial q^{-1} a = aq\partial + 1\rangle.$$

We introduce the elements:
$$g^\alpha:=(1+(q-q^{-1})\partial a), \quad g^\beta:=(1+(q-q^{-1})a\partial).$$
\begin{proposition} We have the relations:
\begin{enumerate}
 \item $g^\alpha \partial = \partial g^\beta$,
 \item $g^\beta a = a g^\alpha$,
 \item $g^\alpha a = q^2a g^\alpha$,
 \item $g^\alpha\partial = q^{-2} \partial g^\alpha.$
\end{enumerate}
\end{proposition}

\begin{proof}
Items (1) and (2) are self-evident.  For (3), we compute:
\begin{align*} g^\alpha a &= (1+(q-q^{-1})\partial a) a\\
 &= (1 + (q-q^{-1})(q^2 a\partial + q))a\\
 &= q^2(1+(q-q^{-1})a\partial)a\\
 &= q^2a(1+(q-q^{-1})\partial a)\\
 &=q^2a g^\alpha,
\end{align*}
as desired.  The computation for (4) is similar to (3).
\end{proof}
\begin{remark} We note in passing that (3) and (4) are special cases of Corollary \ref{qcentral}, which is proven independently.
\end{remark} 
\begin{definition}We let $\cD_q(e)^\circ$ denote the non-commutative localization of $\cD_q(e)$ at the multiplicative Ore set $S:=\{g_\alpha^kg_\beta^l \,\, | \,\, k,l\in \ZZ_{\geq 0}\}$.  \end{definition}
\begin{definition-proposition}\label{easycasenoloop}There exists a unique isomorphism:
\begin{align*}\mathcal{F}: \cD_q(e)^\circ\to\cD_q({e^\vee})^\circ,\\
a \mapsto \partial,\quad \partial \mapsto -ag_\alpha^{-1}.\end{align*}
\end{definition-proposition} 

\begin{proof} Clearly we have a homomorphism $\mathcal{F}: T(\Mat(e)\oplus \Mat(e^\vee))\to \cD_q(e)^\circ$ given on generators as above.  We have only to check that the relations defining $\cD_q(e)$ are mapped to zero by $\mathcal{F}$.
We compute:
\begin{align*}
\mathcal{F}(\partial q^{-1} a - aq\partial - 1) &= -ag_\alpha^{-1}q^{-1}\partial  +\partial q a g_\alpha^{-1} - 1\\
&= -qa\partial g_\alpha^{-1}  + q\partial a g_\alpha^{-1} - 1\\
&= q(\partial a - a\partial)g_\alpha^{-1}- 1\\
&= (1+(q-q^{-1})\partial a)g_\alpha^{-1} - 1 =0,
\end{align*}
as desired.
\end{proof}

\subsubsection{General case: } $e=\overset{n}{\bullet}\to \overset{m}{\bullet}$.
Following the notation of Section \ref{RTTpres} equation \eqref{RTTnot}, we introduce the matrices:
$$g^\alpha:=(I + (q-q^{-1})DA), \quad g^\beta:=(I+(q-q^{-1})AD).$$ 
\begin{proposition}\label{manyrelns} We have the relations:
\begin{enumerate}
 \item $g^\alpha D = D g^\beta$,
 \item $g^\beta A = A g^\alpha$,
 \item $g^\alpha_1R^\beta D_2 = (R^\beta)_{21}^{-1}D_2g^\alpha_1,$
 \item $g^\beta_1(R^\alpha)_{21}^{-1}A_2 = R^\alpha A_2g^\beta_1$,
 \item $g^\beta_1D_2R^\alpha_{21} = D_2(R^\alpha)^{-1}g^\beta_1,$
 \item $g^\alpha_1 A_2(R^{\beta})^{-1} = A_2 R^\beta_{21}g^\alpha_1,$
 \item $g^\beta_1g^\alpha_2 = g^\alpha_2 g^\beta_1$.
\end{enumerate}
\end{proposition}
 \begin{proof}
Items (1) and (2) are self-evident.  For (3), we compute:

\begin{align*}
g^\alpha_1 R^\beta D_2 &= (I+ (q-q^{-1})D_1A_1)R^\beta D_2\\
&= R^\beta D_2+ (q-q^{-1})D_1D_2(R^{\alpha})^{-1}A_1 - (q-q^{-1})D_1\Omega\\
&= R^\beta D_2 + (q-q^{-1})(R^\beta)_{21}^{-1}D_2D_1A_1 - (q-q^{-1})D_1\Omega\\
&= (R^\beta-(q-q^{-1})\Omega^\beta)D_2 + (q-q^{-1})(R^\beta)_{21}^{-1}D_2D_1A_1\\
&= (R^\beta)_{21}^{-1}D_2(I+(q-q^{-1})D_1A_1)\\
&= (R^\beta)_{21}^{-1}D_2 g^\alpha_1.
\end{align*}
Similar computations prove (4)-(6).  For (7), we compute:

\begin{align*} \frac{[g_1^\beta,g_2^\alpha]}{(q-q^{-1})^{2}} &= A_2D_2D_1A_1 - D_1A_1A_2D_2\\
&= A_2R_{21}^\beta \underbrace{(R^\beta)_{21}^{-1} D_2D_1}_{I(e^\vee)}A_1 - D_1A_1A_2D_2\\
&= \underbrace{A_2R_{21}^\beta D_1}_{I(e,e^\vee)}\underbrace{D_2 (R^\alpha)^{-1}A_1}_{I(e,e^\vee)} - D_1A_1A_2D_2\\
&= (D_1(R^\alpha)_{21}^{-1}A_2 -\Omega)( A_1R^\beta D_2 + \Omega) -D_1A_1A_2D_2\\
&= \underbrace{D_1(R^\alpha)_{21}^{-1}A_2A_1R^\beta D_2}_{\textrm{cancel this}} - \Omega A_1R^\beta D_2 + D_1(R^\alpha)_{21}^{-1}A_2\Omega - 1 - \underbrace{D_1A_1A_2D_2}_{\textrm{with this}}\\
&= \Omega(D_2(R^\alpha)^{-1}A_1 -A_1R^\beta D_2 - \Omega)\\
&=0.
\end{align*}

\end{proof}

\begin{definition}  We let $\cD_q(e)^\circ$ denote the non-commutative localization of $\cD_q(e)$ at the quantum determinant $\det_q$ of the matrices $g_\alpha$ and $g_\beta$.
\end{definition}
In Section \ref{q-moment}, Corollary \ref{qcentral} (which is independent of the present section), we prove that the powers of $\det_q$ form a multiplicative Ore set in $\cD_q(e)$, so that the localization is straightforward to construct  (in particular, $\cD^\circ_q(e)$ gives rise to a flat deformation of the localized cotangent bundle to $\Mat_\dd(e)$, which appears in \cite{C-BS}.

\begin{definition-proposition}\label{gencasenoloop} There exists a unique isomorphism:
\begin{align*}\mathcal{F}: \cD_q(e)^\circ \to \cD_q(e^\vee)^\circ,\\
 A\mapsto D, D\mapsto -A(g^{\alpha(e^\vee)})^{-1}.
\end{align*}
\end{definition-proposition}

\begin{proof}
Clearly we have a homomorphism $\mathcal{F}:T(\Mat(e)\oplus\Mat(e^\vee))\to \cD(e)^\circ$ given on generators by the above formula.  We have to check that the relations defining $\cD_q(e)$ are mapped to zero by $\mathcal{F}$. In the formulas below, for each edge $e\in E$, and its adjoint edge $e^\vee \in E^\vee,$ we abbreviate $\alpha=\alpha(e)=\beta(e^\vee)=\beta^\vee$, $\beta=\beta(e)=\alpha(e^\vee)=\alpha^\vee$.  We first compute the image of the relations between the $a(e)^i_j$:

$$\mathcal{F}(R^\alpha A_2A_1 - A_1A_2R^\beta_{21}) = R^{\beta^\vee} D_2D_1 - D_1D_2 R_{21}^{\alpha^\vee} = 0.$$

Next, we compute the image of the relations between the $\partial(e)^i_j$:

\begin{align*}S&:=\mathcal{F}(R^\beta D_2D_1 - D_1D_2 R^\alpha_{21})\\
&=R^{\alpha^\vee} \underbrace{A_2(g^{\alpha^\vee}_2)^{-1}}_{\textrm{\ref{manyrelns} (2)}}A_1(g^{\alpha^\vee}_1)^{-1}  - A_1(g^{\alpha^\vee}_1)^{-1}\underbrace{A_2(g^{\alpha^\vee}_2)^{-1}}_{\textrm{\ref{manyrelns} (2)}}R_{21}^{\beta^\vee}\\
&=R^{\alpha^\vee} (g^{\beta^\vee}_2)^{-1}\underbrace{A_2A_1}_{I(e)}(g^{\alpha^\vee}_1)^{-1}  - A_1(g^{\alpha^\vee}_1)^{-1}(g^{\beta^\vee}_2)^{-1}A_2 R_{21}^{\beta^\vee}\\
&=\underbrace{R^{\alpha^\vee} (g^{\beta^\vee}_2)^{-1}R^{\alpha^\vee}_{21}A_1}_{\ref{manyrelns}\, (5)}\underbrace{A_2(R^{\beta^\vee})^{-1} (g^{\alpha^\vee}_1)^{-1}}_{\ref{manyrelns}\,(6)}  - A_1(g^{\beta^\vee}_1)^{-1}(g^{\alpha^\vee}_2)^{-1}A_2 R_{21}^{\beta^\vee}\\
&=A_1g_2^{-1}g_1^{-1}A_2R^\alpha_{21}  - A_1g_1^{-1}g_2^{-1}A_2 R_{21}^{\alpha}\\
&=0, \end{align*}
by part (7) of Proposition \ref{manyrelns}.

Finally, to compute the image, $\mathcal{F}(D_2R^{-1}A_1-A_1RD_2 - \Omega)$, of the cross relations, we flip tensor factors, and compute:
\begin{align*}\mathcal{F}(D_1(R_{21})^{-1}A_2-A_2R_{21}D_1 - \Omega)
&=-A_1\underbrace{g_1^{-1}R_{21}^{-1}D_2}_{\ref{manyrelns}\,(3)} + D_2R_{21}A_1g_1^{-1} - \Omega\\
&= (D_2R_{21}A_1-\underbrace{A_1RD_2}_{I(e,e^\vee)})g_1^{-1} - \Omega\\
&= (\Omega + D_2(\underbrace{R_{21}-R^{-1}}_{\textrm{Hecke reln.}})A_1)g_1^{-1} - \Omega\\
&= (\Omega + D_2(q-q^{-1})\Omega A_1)g_1^{-1} - \Omega\\
&=\Omega\left((I+(q-q^{-1})D_1A_1)g_1^{-1} - I\right)\\
&=0,
\end{align*}
by definition of  $g_1$.  
\end{proof}
\subsection{Braided Fourier transform on $\cD_q(e)$ when $e$ is a loop}
\subsubsection{Easy case:} $e=\overset{1}{\bullet}\rightloop$.  We again consider the $d_v=1$ case first for the sake of clarity, before moving on to the general situation.  In this case, we have:
$$D_q(e)=\CC\langle \partial, a\rangle \Big/ \langle a\partial = q^2\partial a\rangle.$$
\begin{definition} We let $D_q(e)^\circ$ denote the noncommutative localization at the multiplicative Ore set $S:=\{a^k\partial^l \,\,|\,\, k,l\in \ZZ_{\geq 0}\}$.
\end{definition}

\begin{definition-proposition}\label{easycaseloop} There exists a unique homomorphism:
$$\mathcal{F}:D_q(e)^\circ \to D_q(e^\vee)^\circ,$$
$$a \mapsto \partial, \partial\mapsto \partial a^{-1}\partial^{-1}.$$
Moreover, $\mathcal{F}$ is an isomorphism.
\end{definition-proposition}
\begin{proof} Clearly we have a homomorphism $\mathcal{F}:T(\Mat(e)\oplus\Mat(e^\vee))\to \cD(e)^\circ$ given on generators by the formulas above.  We have to check that the relations defining $\cD_q(e)$ are mapped to zero by $\mathcal{F}$.  We compute:

$$\mathcal{F}(a\partial - q^2\partial a) = \partial(\partial a^{-1}\partial^{-1}) - q^2 (\partial a^{-1} \partial^{-1})\partial= q^2\partial a^{-1}-q^2\partial a^{-1}=0.$$
\end{proof}

\subsubsection{General case:} $e=\overset{n}{\bullet}\rightloop$.
\begin{definition} We let $D_q(e)^\circ$ denote the non-commutative localization at the quantum determinant $\operatorname{det}_q$ of the matrices $D$ and $A$.
\end{definition}

It is well-known that the powers of $\operatorname{det}_q$ form a multiplicative Ore set in $\cD_q(e)$, so that the localization is straightforward.

\begin{definition-proposition}\label{gencaseloop}
There exists a unique isomorphism:
$$\mathcal{F}: \cD_q(e)^\circ\to\cD_q(e)^\circ,$$
$$A \mapsto D, D\mapsto DA^{-1}D^{-1}.$$
\end{definition-proposition}
\begin{proof}
Clearly we have a homomorphism $\mathcal{F}:T(\Mat(e)\oplus\Mat(e^\vee))\to \cD(e)^\circ$ given on generators by the formulas above.  We have to check that the relations defining $\cD_q(e)$ are mapped to zero by $\mathcal{F}$.  Clearly the relations between the $a(e)^i_j$ are sent to zero, as $\mathcal{F}(A)=D$ still satisfies the reflection equations.  We compute the image of the relations between the $\partial(e)^i_j$.

\begin{align*}\mathcal{F}(D_2R_{21}D_1R) &= D_2A_2^{-1}D_2^{-1}R_{21}D_1A_1^{-1}D_1^{-1}R\\
&= D_2A_2^{-1}\underbrace{D_2^{-1}R_{21}D_1R}_{I(e^\vee)}R^{-1}A_1^{-1}D_1^{-1}R\\
&= D_2\underbrace{A_2^{-1}R_{21}D_1R}_{I(e,e^\vee)}\underbrace{D_2^{-1}R^{-1}A_1^{-1}}_{I(e,e^\vee)}D_1^{-1}R\\
&= D_2R_{21}D_1\underbrace{R_{21}^{-1}A_2^{-1}R^{-1}A_1^{-1}}_{I(e)}\underbrace{R_{21}^{-1}D_2^{-1}R^{-1}D_1^{-1}R}_{I(e^\vee)}\\
&= \underbrace{D_2R_{21}D_1R}_{I(e^\vee)}R^{-1}A_1^{-1}\underbrace{R_{21}^{-1}A_2^{-1}R^{-1}D_1^{-1}R_{21}^{-1}}_{I(e,e^\vee)}D_2^{-1}\\
&= R_{21}D_1\underbrace{RD_2R^{-1}A_1^{-1}}_{I(e,e\vee)}D_1^{-1}R_{21}^{-1}A_2^{-1}D_2^{-1}\\
&= R_{21}D_1A_1^{-1}\underbrace{R_{12}D_2R_{21}D_1^{-1}R_{21}^{-1}}_{I(e^\vee)}A_2^{-1}D_2^{-1}\\
&= R_{21}D_1A_1^{-1}D_1^{-1}R_{12}D_2A_2^{-1}D_2^{-1}\\
&= \mathcal{F}(R_{21}D_1RD_2).
\end{align*}

Finally, we compute the image of the cross relations.  We find:
\begin{align*}\mathcal{F}(A_1RD_2)
&=\underbrace{D_1R D_2}_{I(e^\vee)}A_2^{-1}D_2^{-1}\\
&=R D_2\underbrace{RD_1R^{-1}A_2^{-1}}_{I(e,e^\vee)}D_2^{-1}\\
&=R D_2A_2^{-1}\underbrace{R_{21}D_1RD_2^{-1}}_{I(e,e^\vee)}\\
&=R D_2A_2^{-1}D_2^{-1}R_{21}D_1R\\
&= \mathcal{F}(RD_2R_{21}A_1R).
\end{align*}\edit{Need to check that map preserves Ore sets.}
\end{proof}

\subsection{Independence of $\cD_q(\Mat_\dd(Q))^\circ$ on the orientation of $Q$.}
For a quiver $Q$, and $e\in E$, let $\tau_e(Q)$ denote the quiver obtained from $Q$ by reversing the orientation of $e$.

Let $Q_1$ and $Q_2$ be quivers whose underlying undirected graphs are isomorphic.  Choose an isomorphism, by which we can identify the sets $V_1$, $V_2$ of vertices, and $\widetilde{E}_1$, $\widetilde{E}_2$ of undirected edges.
We have the following:

\begin{theorem}
Let $e_1,\ldots, e_n$ be a sequence of edges of $Q_1$, such that $\tau_{e_n}\cdots\tau_{e_1}(Q_1)\cong Q_2$ as oriented graphs.  Then there is an induced isomorphism,
$$\cD_q(\Mat_\dd(Q_1))^\circ \cong \cD_q(\Mat_\dd(Q_2))^\circ.$$
\end{theorem}

\begin{proof}
Clearly, it suffices to assume that the orientations on $Q_1$ and $Q_2$ differ at exactly one edge.  In this case, the isomorphism $\cD_q(e)\to \cD_q(e^\vee)$ constructed in the previous section can be extended to an isomorphism $\cD_q(\Mat_\dd(Q_1))\to\cD_q(\Mat_\dd(Q_2))$, as the relations between $\cD_q(e)$ (resp, $\cD_q(e^\vee)$) and the rest of $\cD_q(\Mat_\dd(Q_1))$ (resp, $\cD_q(\Mat_\dd(Q_2))$) are just the tensor product relations, which are preserved by $\mathcal{F}$, which is a morphism in $\cC$.
\end{proof}

\section{Construction of the $q$-deformed quantum moment map}\label{q-moment}
In this section we construct the $q$-analog of the moment map in the classical geometric construction of the quiver variety.
\subsection{Bialgebras and Hopf algebras in braided tensor categories}
We recall some basic constructions involving Hopf algebras in braided tensor categories, which we will use later.
\begin{definition} A bialgebra in $\cC$ is a 5-tuple,
$$(A\in\cC,\,\,\mu:A\ot A\to A,\,\,\eta:\mathbf{1}\to A,\,\,\Delta: A\to A\ot A,\,\, \epsilon:A\to\mathbf{1}),$$
such that $(A,\mu,\eta)$ is a unital algebra in $\cC$, $(A,\Delta,\epsilon)$ is a co-unital coalgebra in $\cC$, $\Delta$ is a homomorphism to the tensor product algebra $A\ot A$.  Homomorphisms are defined in the obvious way, and we denote by $\cC$-biAlg the category of bialgebras in $\cC$.
\end{definition}
\begin{definition} A Hopf algebra in $\cC$ is a bialgebra in $\cC$, with a (necessarily unique) convolution inverse $S$ to the identity, called the antipode: either composition,
$$S\ast \id: A \xrightarrow{\Delta} A\ot A \xrightarrow{S\ot \id} A\ot A \xrightarrow{\mu} A,$$
$$\id\ast S: A \xrightarrow{\Delta} A\ot A \xrightarrow{\id \ot S} A\ot A \xrightarrow{\mu} A,$$
coincides with the convolution unit $\eta\circ\epsilon:A\to A$.  We define the category \mbox{$\cC-$Hopf-Alg} as the full subcategory of $\cC$-biAlg consisting of bialgebras with antipode.
\end{definition}

Let $H$ be a Hopf algebra (in $\Vect$), $A$ be an algebra, and $\phi:H\to A$ be a homomorphism of algebras.  To simplify notation, we omit the explicit application of $\phi$ here and in the definitions to follow.  $H$ acts on $A$ via the induced adjoint action, $h\rhd a= h_{(1)}aS(h_{2}) \in A$.  For $\cC$-Hopf-Alg, there is an analogous construction:

\begin{definition}  Let $H\in\cC$-Hopf-Alg, and let $A\in \cC$-Alg.  Let $\phi:H\to A$ be a homomorphism of $\cC$-algebras.  The regular action of $H\ot H$ on $A$ is defined by:
$$\operatorname{act}_2:H \ot H\ot A \xrightarrow{\id\ot \sigma_{H,A}} H\ot A\ot H \xrightarrow{\id_H\ot\id_A\ot S} H\ot A\ot H \xrightarrow{\mu\circ(\id\ot\mu)} A.$$
The adjoint action of $H$ on $A$ is given by 
$$\ad: H\ot A\xrightarrow{\Delta\ot \id} H\ot H\ot A \xrightarrow{\operatorname{act}_2} A.$$
\end{definition}
It is a standard exercise to check that these are indeed actions, i.e. that
$$\ad\circ(\mu_H\ot\id_A) = \ad\circ(\id_H\ot\ad):H\ot H\ot A\to A.$$

\subsection{Hopf algebra of matrix coefficients}
For a locally finite braided tensor category $\cD$, we have its algebra $A(\cD)$ of matrix coefficients, whose general construction dates back to work of Lyubashenko and Majid \cite{LyMa}, \cite{Ma}.  We recall the construction here.

We have the functor of tensor product,\glossary{$\mathcal{T}$}
$$\mathcal{T}:\cD\bt\cD \to \cD,$$ $$V\bt W\mapsto V\ot W.$$
The braiding endows $\mathcal{T}$ with the structure of a tensor functor:
$$J: \mathcal{T}(X\bt U)\ot \mathcal{T}(V\bt W) = X\ot U\ot V \ot W \xrightarrow{\sigma_{U,V}}X\ot V\ot U \ot W = \mathcal{T}((X\bt U)\ot (V\bt W)).$$
$\mathcal{T}$ has a right adjoint $\mathcal{T}^\vee$ taking values in the Ind-category of $\cD\bt\cD$.  We define $A(\cD):= \mathcal{T}^\vee (\mathbf{1_\cD})$, and call it the algebra of matrix coefficients (for reasons which will become clear below).  $A(\cD)$ is thus defined uniquely, up to canonical isomorophism, as the representing object for the functor of co-invariants,
 $$\Hom_{\cD\bt\cD}(-\bt -,A(\cD))\cong \Hom_{\cD}(-\ot-,\mathbf{1}).$$
This description allows us to construct $A(\cD)$ explicitly as an Ind-algebra in $\cD\bt\cD$.  We let $\widetilde{A}(\cD)$ be the sum over all objects of $\cD$,
$$\widetilde{A}(\cD):=\bigoplus_{V\in\cD} V^*\bt V,$$
and let $A(\cD)$ be the quotient $\widetilde{A}(\cD)/Q$, where $Q$ denotes the sum over all morphisms,
$$Q:= \sum_{\phi:V\to W} \operatorname{im}(\Delta_\phi) \subset \widetilde{A}(\cD), \textrm{ where }$$
$$\Delta_\phi := (\id\bt\phi - \phi^*\bt \id):W^*\bt V\to W^*\bt W\oplus V^*\bt V.$$

To see that $A(\cD)$ does indeed satify the desired universal property, we observe that we have natural isomorphisms:
$$\Hom_{\cD}(X\bt Y,A(\cD))\cong \Hom(X,Y^*) \cong \Hom(X\ot Y,\mathbf{1}),$$
because we can write any morphism $\phi\in\Hom(Y,V)$ as $\phi\circ\id_Y,$ and can then apply the relations of $Q$ to reduce the sum over all $V$ to the single summand $V=Y$.

We have natural morphisms $i_V: V^*\bt V\to A$, and also $\mathcal{T}(i_v):V^*\ot V\to \mathcal{T}(A)$, for all $V\in \cC$.  We will abuse notation and call $\mathcal{T}(i_v)$ simply by $i_V$ when context is clear.

The algebra structure on $A$ is given on generating objects $V^*\bt V$, $W^*\bt W$ by
$$(V^*\bt V)\ot_2 (W^*\bt W) = V^*\ot W^*\bt V\ot W \xrightarrow{\sigma_{V^*,W^*}} W^*\ot V^*\bt V\ot W \xrightarrow{i_{V\ot W}} A.$$
The algebra structure on $\mathcal{T}(A)$\glossary{$\mathcal{T}(A)$} is given on generating objects $V^*\ot V$, $W^*\ot W$ by
$$ (V^*\ot V)\ot (W^*\ot W) \xrightarrow{\sigma_{(V^*\ot V),W^*}} (W^*\ot V^*)\ot (V\ot W) \xrightarrow{i_{V\ot W}} \mathcal{T}(A).$$
The unit of $A$, (resp. $\mathcal{T}(A)$) is the subspace $\mathbf{1}^*\bt \mathbf{1}$ (resp. $\mathbf{1}\cong \mathbf{1}^*\ot\mathbf{1}$).

\begin{remark}
The adjoint pair of functors ($\mathcal{T},\mathcal{T}^\vee)$ are braided tensor categorical analogs of the restriction and induction functors, $(\operatorname{Res}^{G\times G}_G,\operatorname{Ind}^{G\times G}_G)$, of finite groups, and the construction given above is analogous to constructing the $G-G$-bimodule $\CC[G]$ as $\operatorname{Ind}^{G\times G}_G \CC$.  
\end{remark}

\begin{remark}
In case $\cD=U$-mod, for some quasi-triangular Hopf algebra $H$, the algebra $A(\cD)$ identifies as a vector space with the subspace of $H^*$ spanned by functionals $c_{f,v},$ for  $v\in V, f\in V^*$ defined by $c_{f,v}(h):=f(hv)$.  Choosing a basis $v_1,\ldots v_n$ and its dual basis $f_1,\ldots, f_n$, one has the functionals $c_{f_i,v_j}(h)$, which are the $i,j$th matrix entry of the map $H\to \Mat_n(\CC)$ of the representation $V$.
\end{remark}

\begin{definition} $\mathcal{T}(A(\cD))$ becomes a Hopf algebra in $\cD$ with coproduct, counit, and antipode defined on each subspace $V^*\ot V$ by:
$$\Delta|_{V^*\ot V}: V^*\ot V \xrightarrow{\id\ot\coev\ot\id} V^*\ot V\ot V^*\ot V \xrightarrow{i_V\ot i_V} \mathcal{T}(A(\cD))\ot \mathcal{T}(A(\cD)),$$
$$\epsilon|_{V^*\ot V}: V^*\ot V \xrightarrow{\ev} \mathbf{1},$$
$$S|_{V^*\ot V}:V^*\ot V \xrightarrow{\sigma_{V^*\ot V}} V\ot V^* \xrightarrow{\theta_V\ot \id} V^{**}\ot V^* \xrightarrow{i_{v_{^*}}} \mathcal{T}(A(\cD)).$$
\end{definition}

For the category $\cC:=\bt_{v\in V}\cC_v$ of Section \ref{sec:QuivNot}, we have $A(\cC):=\bt_{v\in V}A(\cC_v)$, and $\mathcal{T}(A(\cC)):=\bt_{v\in V}\mathcal{T}(A(\cC_v))$, which becomes a $\cC$-Hopf algebra with structure morphisms defined diagonally.\glossary{$A(\cC)$, $\mathcal{T}(A(\cC))$}

\subsubsection{The quantum determinant $\det_q$}
When $\cD$ is the braided tensor category of type I $U_q(\mathfrak{gl}_N)$-modules, the algebra $\mathcal{T}(A)$ of the previous section contains a central element called the quantum determinant.  It is defined as follows:\glossary{$\det_q$}

\begin{definition} The quantum determinant $\det_q$ is the unique generator of the one-dimensional subspace $\mathcal{T}((\Lambda^N_q(\CC^N))^*\bt \Lambda^N_q(\CC^N))$, with normalization $ev_{\Lambda^N_q}(\det_q)=1$.
\end{definition}

\begin{proposition}
The element $\det_q$ is central and group-like in $\mathcal{T}(A(\mathcal{D}))$.
\end{proposition}
\begin{proof}
That $\det_q$ is group-like is clear from the fact that $\Lambda^N_q(\CC^N))$ is one-dimensional.  It's centrality follows from the fact that $\sigma_{\Lambda^N_q(\CC^N)),V}$ is a scalar matrix for any $V$.
\end{proof}

\begin{remark}
The construction of $\det_q$ above does not yield a clear formula for $\det_q$ in terms of the standard generators $a^i_j$.  We have been unable to find such a formula in the literature, for the reflection equation algebra, although a formula for the corresponding element in the FRT algebra is well-known.
\end{remark}

\subsubsection{Explicit presentation of $\mathcal{T}(A(\cC_v))$.}
We have the following well-known presentation for $\mathcal{T}(A(\cC_v))$.
\begin{theorem} We have an isomorphism:
$$\mathcal{T}(A(\cC_v)) \cong \mathcal{T}(A(\cC_v))^+[(\operatorname{det}_q)^{-1}], \textrm{ where }$$
\begin{equation}\label{refleqns}\mathcal{T}(A(\cC_v))^+:=\langle l^i_j,\,\, i,j=1,\ldots, d_v \,\, | \,\, R^{ij}_{kl}l^l_m R^{mk}_{no}l^o_p=l^i_lR^{lj}_{km}l^m_oR^{ok}_{np}\rangle.\end{equation}
\end{theorem}

In particular, there is a well-known isomorphism of algebras,
$$\kappa:\mathcal{T}(A(\cC_v))\to U_q'(\mathfrak{gl}_{d_v})$$
$$ l^i_j \mapsto \widetilde{l}_{ij}.$$
We note in passing that $\kappa(\mathcal{T}(A(\cC))^+)=U^+$.  Henceforth, we identify $\mathcal{T}(A(\cC_v))$ with $U'_q(\mathfrak{gl}_{d_v})$ and $\mathcal{T}(A(\cC_v))^+$ with $U^+$ via the isomorphism $\kappa$.\glossary{$\kappa$}

\subsection{Quantum moment map for $\cD_q(e)$ when $e$ is not a loop}
In the next two sections, we construct quantum moment maps, $\mu_v^e:U_v\to\cD_e$ for each edge $e\in E$, and $v=\alpha(e),\beta(e)$.  As might be expected, the construction is quite different depending on whether or not $e$ is a loop.  As such, we treat the two cases in different sections.

First, we recall from \cite{VV} the notion of a quantum moment map for a coideal subalgebra, a mild generalization of that in \cite{L}.  Let $H$ be a Hopf algebra, with $H'$ a left coideal subalgebra; that is, $H'$ is a subalgebra, and $\Delta(H')\subset H\otimes H'$.  Then a  moment map for an $H$-algebra $A$ is a homomorphism $\mu:H'\to A$, such that
$$\mu(h)a= (h_{(1)}\rhd a)\mu(h_{(2)}),$$
where we denote the action of $h\in H$ on $a\in A$ by $h\rhd a$ to distinguish it from the multiplication in $A$.

\begin{definition-proposition}\label{edgemomdef}\glossary{edge map, $\mu^e_v$} Let $e\in E$, and $v=\beta(e)\neq\alpha(e)$.  The \emph{edge map} $\mu^e_v:U_v^+\to \cD_e$ given on generators by:\footnote{We have set $t=1$ in the definition of $\cD_q(\Mat_d(Q))$, for ease of notation (see Remark \ref{teq1}).  It is easily checked that defining $\mu^e_v(l^i_j):=\delta^i_j +t(q-q^{-1})\partial^i_ka^k_j$ yields a moment map for other choices of $t$.  This will be needed in Section \ref{sDAHA}.}
$$\mu^e_v(l^i_j) = (\delta^i_j+(q-q^{-1})\sum_k\partial^i_ka^k_j),$$
defines a homomorphism of algebras in $\cC$.\label{momentbeta}
\end{definition-proposition}
\begin{proof}
Following the notation of Section \ref{sec:QG}, we let $M$ denote the matrix:
$$M:= \sum_{i,j} \mu^e_v(l^i_j) E^j_i.$$
We have $M= I + (q-q^{-1})DA$.  We need to show that the elements $\mu^e_v(l^i_j)\in\cD_e$ satisfy the reflection equation relations \eqref{refleqns}.  We compute, in matrix notation:
\begin{align*}
M_2R_{21}M_1R_{12} &= (I + (q-q^{-1})D_2A_2)R_{21}(I+(q-q^{-1})D_1A_1)R_{12}\\
&=R_{21}R_{12} + (q-q^{-1})(D_2A_2\underbrace{R_{21}R_{12}}_{\textrm{Hecke reln.}} + R_{21}D_1A_1R_{12})\\&\phantom{===} + (q-q^{-1})^2D_2\underbrace{A_2R_{21}D_1}_{I(e,e^\vee)}A_1R_{12}\\
&=R_{21}R_{12} + (q-q^{-1})(D_2A_2 + (q-q^{-1})\underbrace{D_2A_2\Omega_{12}R_{12}}_{\textrm{cancel this}} + R_{21}D_1A_1R_{12})\\&\phantom{===} + (q-q^{-1})^2(D_2D_1\underbrace{R_{21}^{-1}A_2A_1R_{12}}_{I(e)} - \underbrace{D_2\Omega_{12}A_1R_{12})}_{\textrm{with this}}\\
&=R_{21}R_{12} + (q-q^{-1})(D_2A_2 + R_{21}D_1A_1R_{12}) + (q-q^{-1})^2D_2D_1A_1A_2\\
\end{align*}
On the other hand, we compute:
\begin{align*}
R_{21}M_1R_{12}M_2 &= R_{21}(I+(q-q^{-1})D_1A_1)R_{12}(I+(q-q^{-1})D_2A_2)\\
&= R_{21}R_{12} + (q-q^{-1})(R_{21}D_1A_1R_{12} + \underbrace{R_{21}R_{12}}_{\textrm{Hecke reln.}}D_2A_2) \\
&\phantom{===} + (q-q^{-1})^2R_{21}D_1\underbrace{A_1R_{12}D_2}_{I(e,e^\vee)}A_2\\
&= R_{21}R_{12} + (q-q^{-1})(D_2A_2 +  (q-q^{-1})\underbrace{\Omega_{12}R_{12}D_2A_2}_{\textrm{cancel this}} + R_{21}D_1A_1R_{12}) \\
&\phantom{===} + (q-q^{-1})^2(\underbrace{R_{21}D_1D_2R_{12}^{-1}}_{I(e^\vee)}A_1A_2 - \underbrace{R_{21}D_1\Omega_{12}A_2}_{\textrm{with this}})\\
&= R_{21}R_{12} + (q-q^{-1})(D_2A_2 + R_{21}D_1A_1R_{12}) +  (q-q^{-1})^2D_2D_1A_1A_2\\
&=M_2R_{21}M_1R_{12},
\end{align*}
as desired.  Thus the homomorphism $\mu^e_v$ is well defined.
\end{proof}

\begin{proposition}\label{mommap}
Let $v=\beta(e)\neq\alpha(e)$.  Regard $\mu_v^e$ above as a map from $U^+$ via the isomorphism $\kappa$.  Then $\mu_v^e$ is a quantum moment map: 
$$\mu^e_v(x)y=(x_{(1)}\rhd y) \mu^e_v(x_{(2)}),$$ for all $x\in U^+, y\in\cD^\circ_e$.
\end{proposition}
\begin{proof}It suffices to check this on the generators $\widetilde{l}^i_{j}$ of $U^+$, and the generators $a^m_n$, $\partial^o_p$ of $\cD_e$.  By definition of the $U^+$ action on $V$, we have:

\begin{align*}
((\widetilde{l}^i_{j})_{(1)}\rhd a^m_n)\mu^e_v((\widetilde{l}^i_{j})_{(2)}) &=((l^{+i}_{k}S(l^{-l}_{j}))\rhd a^m_n)(\delta^k_l + (q-q^{-1}) \partial^k_oa^o_l)\\
&=R^{qi}_{pk}R^{lp}_{jn}a^m_q(\delta^k_l + (q-q^{-1}) \partial^k_oa^o_l).
\end{align*}

In the matrix notation of Section \ref{sec:QG}, we set
\begin{align*}N&:=\sum_{i,j,n,m}((\widetilde{l}^i_{j})_{(1)}\rhd a^m_n)\mu^e_v((\widetilde{l}^i_{j})_{(2)})E^j_i\ot E^n_m\\
&=\sum_{i,j,n,m}(R^{qi}_{pk}R^{lp}_{jn}a^m_q(\delta^k_l + (q-q^{-1}) \partial^k_oa^o_l)) E^j_i\ot E^n_m.
\end{align*}

Then, we have:
\begin{align*}
N&=A_2R_{21}(I+(q-q^{-1})D_1A_1)R_{12}\\
  &=A_2R_{21}R_{12} + (q-q^{-1})\underbrace{A_2R_{21}D_1}_{I(e,e^\vee)}A_1R_{12}\\
    &=A_2R_{21}R_{12} + (q-q^{-1})(D_1\underbrace{R^{-1}_{21}A_2A_1R_{12}}_{I(e)} - \Omega_{12}A_1R_{12})\\
     &=A_2(\underbrace{R_{21}R_{12} - \Omega_{12}R_{12}}_{\textrm{Hecke reln.}})  + (q-q^{-1})D_1A_1A_2\\
     &=A_2 + (q-q^{-1})D_1A_1A_2\\
     &=M_1A_2
\end{align*}
Comparing matrix coefficients, we find:
$$((\widetilde{l}^i_{j})_{(1)}\rhd a^m_n)\mu^e_v((\widetilde{l}^i_{j})_{(2)}) = \mu^e_v(\widetilde{l}^i_{j})a^m_n,$$
as desired.  The computation for $\partial^o_p$ is similar.
\end{proof}

\begin{corollary} \label{qcentral}The image $\mu^e_v(\operatorname{det}_q)$ of the quantum determinant in $U^+$ satisfies the equation:
$$\mu_v^e(\operatorname{det}_q) a^I\partial_J = q^{2(|J|-|I|)}a^I\partial_J\mu_v^e(\operatorname{det}_q),$$
\end{corollary}
\begin{proof}
Recall that $\operatorname{det}_q$ is grouplike in $U^+$.  Thus the moment map condition reads:
$$\mu_v^e(\operatorname{det}_q)a^I\partial_J=(\operatorname{det}_q\rhd a^I\partial_J)\mu_v^e(\operatorname{det}_q).$$
The element $\operatorname{det}_q$ acts on $V\in\cC_v$ by the scalar $q^2$, and $V^*\in\cC_v$ by the scalar $q^{-2}$, so the claim follows.
\end{proof}
\begin{proposition}\label{momentalpha} Let $e\in E$, and $v=\alpha(e)\neq\beta(e)$.  The elements $\overline\mu^e_v(l^i_j)$,
$$\overline\mu^e_v(l^i_j) = (\delta^i_j+(q-q^{-1})\sum_k a^i_k\partial^k_j),$$
satisfy the relation:
$$\overline{M_2}R_{12}^{-1}\overline{M_1}R_{21}^{-1} = R_{12}^{-1}\overline{M_1}R_{21}^{-1}\overline{M_2},$$
where $\overline{M}$ denotes the matrix:
$$\overline{M}:=\sum_{i,j} \overline{\mu}^e_v(l^i_j)E^j_i.$$
\end{proposition}
\begin{proof}
We observe that the defining relations of $\cD_e$ and $\cD_{e^\vee}$ are related by interchanging each $a^i_j$ with $\partial^i_j$, and replacing $R_\alpha,R_\beta$ with $(R_\alpha)_{21}^{-1}$, $(R_\beta)_{21}^{-1}$, so that this relation follows from Definition-Proposition \ref{momentbeta} .
\end{proof}
\begin{corollary} \label{oreset}
Let $v=\beta\neq\alpha$ (resp, $v=\alpha\neq\beta$).  The powers of the $q$-determinant in the variables $\mu^v_e(l^i_j)$ (resp, $\overline{\mu}^e_v(l^i_j)$) form a multiplicative Ore set.
\end{corollary}
\begin{proof}
This follows as in Corollary \ref{qcentral}.
\end{proof}

\begin{definition}
The localized edge differential operator algebra $\cD_q(e)^\circ$ is the localization of $\cD_q(e)$ at the multiplicative Ore sets generated by the $q$-determinants in the elements $\mu^e_\beta(l^i_j)$ and $\overline{\mu}^e_\alpha(l^k_l)$.
\end{definition}

\begin{definition-proposition}  Let $e\in E$, and $v=\alpha(e)\neq \beta(e)$.  The \emph{edge map} $\mu_v^e:U_v^+\to \cD_e^\circ$ given on generators by:\footnote{For $t\neq 1$, we set \mbox{$\mu_v^e(l^i_j)=(\delta^i_j+t(q-q^{-1})a^i_k\partial^k_j)^{-1}$} instead (see Definition \ref{edgemomdef}, Remark \ref{teq1}).}
$$\mu_v^e(l^i_j)=(\delta^i_j+(q-q^{-1})\sum_k a^i_k\partial^k_j)^{-1},$$
defines a homomorphism of algebras in $\cC$.
\end{definition-proposition}
\begin{proof}
The entries of the inverse matrix in the definition lie in the localized algebra $\cD_q(e)^\circ$, where we have inverted the $q$-determinant.  That $\mu_v^e$ defines a homomorphism follows from Proposition \ref{momentalpha}, by taking the inverses of both sides.
\end{proof}

\begin{definition} The edge maps $\mu^e_\alpha$ and $\mu^e_\beta$ extend uniquely to homomorphisms $\mu^e_\alpha:U_\alpha\to\cD_e^\circ$ and $\mu^e_\beta:U_\beta\to\cD_e^\circ$.
\end{definition}
We will henceforth refer only to this extended homomorphism, and not its restriction to $U^+$.

\subsection{Quantum moment map for $\cD_q(e)$ when $e$ is a loop}

\begin{definition} Let $v=\alpha(e)=\beta(e)$.  The localized edge algebra $\cD_q^\circ(e)$ is the localization of $\cD_q(e)$ at the $q$-determinants in the variables $a^i_j$ and $\partial^i_j$ of $\cD_q(e)$.
\end{definition}

\begin{definition-proposition} There is a unique homomorphism,
$$\mu^e_v:\mathcal{T}(A(C_v))\to\cD_q(e),$$
$$l^i_j \mapsto (DA^{-1}D^{-1}A)^i_j$$
\end{definition-proposition}
\begin{proof}

First, we claim that there is a unique homomorphism of algebras in $\cC$,
$$\phi: \mathcal{T}(A(\cC_v))\ot \mathcal{T}(A(\cC_v))\to \cD_q(e)$$
$$(l^i_j\ot l^k_l) \mapsto (DA^{-1}D^{-1})^i_j a^k_l .$$
Once we have constructed $\phi$, we can simply define $\mu^e_v:=\phi\circ\Delta$.

An algebra homomorphism $\phi =f\ot g$ out of $\mathcal{T}(A(\cC_v))\ot \mathcal{T}(A(\cC_v))$ is the same as a pair $f,g$ of algebra homomorphism out of $\mathcal{T}(A(\cC_v))$, such that the images of $f$ and $g$ braided-commute.  That is, we require the following relations on $\mathcal{T}(A(\cC_v))\ot \mathcal{T}(A(\cC_v))$:
$$(1\ot x)(y\ot 1) = r^- y\ot r^+ x,$$

where we use the shorthand $R=r^+\otimes r^-$ (sum is implicit).  On generators $x=l^i_j$, $y=l^k_l$, this condition reads:
$$ (1\ot l^i_j)(l^k_l\ot 1) = \widetilde{R}^{mk}_{jn} R^{op}_{ml}R^{in}_{qr}(R^{-1})^{qs}_{tp}(l^r_s\ot l^t_o),$$
or equivalently,
\begin{equation*}(1\ot l^i_j)R^{jn}_{mk}(l^k_l\ot 1) = R^{in}_{qr}(l^r_s\ot 1)(R^{-1})^{qs}_{tp}(1\ot l^t_o) R^{op}_{ml}.\end{equation*}

Thus the condition we require on $f$ and $g$ is:
$$g(l^i_j)R^{jn}_{mk}f(l^k_l) = R^{in}_{qr}f(l^r_s)(R^{-1})^{qs}_{tp}g(l^t_o) R^{op}_{ml}$$
or, in the matrix notation of Section \ref{sec:QG}:
\begin{equation}G_1RF_2=RF_2 R^{-1} G_1R,\label{FG}\end{equation}
where $F$ and $G$ denote the matrices:
$$F:=\sum_{i.j} f(l^i_j)E^j_i, \quad G := \sum_{i.j}g(l^i_j)E^j_i.$$
The maps $f,g: \mathcal{T}(A(\cC))\to\cD_e^\circ$, $f(l^i_j)= (DA^{-1}D^{-1})^i_j$ and $g(l^i_j)=A^i_j$ are each homomorphisms (they are the natural inclusion of $\mathcal{T}(A(\cC))$, and its composition with Fourier transform, respectively).  It remains to check relation \eqref{FG}.  We have:

\begin{align*}RF_2 R^{-1} G_1R &= RDA^{-1}\underbrace{D^{-1}R^{-1}AR}_{I(e,e^\vee)}\\
&= RD\underbrace{A^{-1}RAR}_{I(e)}D^{-1}\\
&= \underbrace{RDRAR}_{I(e,e^\vee)}A^{-1}D^{-1}\\
&= ARDA^{-1}D^{-1}\\
&= G_1RF_2,
\end{align*}
as desired.
\end{proof}

\begin{proposition}\label{mommaploop} Let $v=\beta(e)=\alpha(e)$.  Regard $\mu_v^e$ above as a map from $U^+$ via the isomorphism $\kappa$.  Then $\mu_v^e$ is a quantum moment map for all $x\in U^+, y\in \cD_e^\circ$.
\end{proposition}

\begin{proof}
It suffices to check this on the generators $\widetilde{l}^i_{j}$ of $U^+$, and the generators $a^m_n$, $\partial^o_p$ of $\cD_e$.  By definition of the $U^+$ action on $V$, we have:
\begin{align*}
((\widetilde{l}^i_{j})_{(1)}\rhd a^m_n)\mu^e_v((\widetilde{l}^i_{j})_{(2)}) &=((l^{+i}_{k}S(l^{-l}_{j}))\rhd a^m_n)(DA^{-1}D^{-1}A)^k_l\\
&=(R^{-1})^{mi}_{qs}\widetilde{R}^{tq}_{jx}a^x_oR^{os}_{pk}R^{lp}_{tn}(DA^{-1}D^{-1}A)^k_l.
\end{align*}
Thus, the moment map condition reads:
$$(DA^{-1}D^{-1}A) \ot A = (R^{-1})^{mi}_{qs}\widetilde{R}^{tq}_{jx}a^x_oR^{os}_{pk}R^{lp}_{tn}(DA^{-1}D^{-1}A)^k_l E^j_i \ot E^n_m,$$
or equivalently, moving the $R^{-1}$ and $\widetilde{R}$ to the LHS, and re-writing the RHS in matrix notation:
$$R_{21}D_1A_1^{-1}D_1^{-1}A_1R A_2 = A_2R_{21}D_1A_1^{-1}D_1^{-1}A_1R_{12}.$$

We simplify the RHS:
\begin{align*}RHS &= \underbrace{A_2R_{21}D_1}_{I(e,e^\vee)}A_1^{-1}D_1^{-1}A_1R\\
 &= R_{21}D_1\underbrace{R A_2 R_{21}A_1^{-1}}_{I(e)}D_1^{-1}A_1R\\
 &= R_{21}D_1 A_1^{-1}\underbrace{R A_2 R_{21}D_1^{-1}}_{I(e,e^\vee)}A_1R\\
 &= R_{21}D_1 A_1^{-1}D_1^{-1}R_{21}^{-1}\underbrace{A_2R_{21}A_1R}_{I(e)}\\
  &= R_{21}D_1 A_1^{-1}D_1^{-1}A_1RA_2,\\
\end{align*}
and thus the moment map condition is satisfied.
\end{proof}
\subsection{Quantum moment map for $\cD_q(\Mat_\dd(Q))$}
In the previous section, we defined moment maps $\mu^e_v:U_v\to\cD_e$ for every pair $(e,v)$, with $e$ attached to $v$.  In this section, we combine the edge maps into a homomorphism $\mu^\#_q:A(\cC)\to \cD_q(\Mat_\dd(Q))$, quantizing the moment map defined in the classical case.  First, we have:


%


\begin{lemma} For all $v, v'\in V$ distinct, and for all $e\in E_v, e'\in E_{v'}$, we have:
$$\mu_v^e(l^i_j)\mu_{v'}^{e'}(l^k_l)=\mu_{v'}^{e'}(l^k_l)\mu_v^e(l^i_j).$$
\end{lemma}
\begin{proof}
We claim that, for any $e$ emanating from $v$ and for any $w\neq v$, the image of $\mu^e_v$ is contained in a trivial isotypic component of $\cC_{w}$.  This is obvious if $e$ is a loop, and for $e$ not a loop, it follows from the following, more canonical description of $\mu^e_v:$
$$\mu^e_v(l^i_j) = \ev(v^i\ot v_j) + (q-q^{-1})v^i\ot v_j \bt \coev(\mathbf{1}).$$
Since $\cD_q$ is defined as a braided tensor product over its edge algebras $\cD_e$, elements in the image of $\mu^e_v$ commute with those in the image of any $\mu^f_w$, via the braiding.  As the trivial representation braids trivially with any representation, the claim follows.
\end{proof}

\begin{remark}  At this point, we note that the ordering on $\overline{E}$ is not used in any construction, but rather the induced ordering on each $\overline{E_v}$.  This is consistent with similar observations in \cite{C-BS}, \cite{VdB2}.
\end{remark}

\begin{definition}\glossary{$\mu_v$} The vertex moment map $\mu_v^\#: A(\cC_v)\to\cD_q(\Mat_\dd(Q))$ is the composition:
$$\mu_v^\#:A(\cC_v)\xrightarrow{\Delta^{(|E_v|)}}A(\cC)^{\ot |E_v|}\xrightarrow{ \underset{e\in E_v}{\bigotimes}\mu^e_v} \bigotimes_{e\in E_v} \cD_e \subset \cD_q(\Mat_\dd(Q)).$$
\end{definition}

\begin{definition}\glossary{$\mu_q^\#$} The moment map $\mu_q^\#:T(A(\cC))\to \cD_q(\Mat_\dd(Q))$ is the external tensor product,
$$\mu_q^\#:=\underset{v\in V}{\boxtimes}\mu^\#_v: T(A(\cC))\cong \underset{v\in V}{\boxtimes}A(\cC_v)\to\cD_q(\Mat_\dd(Q)).$$
\end{definition}

It follows by Propositions \ref{mommap} and \ref{mommaploop} that $\mu_q^\#$ is indeed a moment map in the sense of \cite{L}.

\section{Construction of the quantized multiplicative quiver variety}\label{finalconstruction}
In this section, we are finally in a position to define the quantized multiplicative quiver variety.  First, we recall certain characters of $A(\cC_v)$, where \mbox{$\cC_v=U_q(\mathfrak{gl}_{d_v})$-lfmod}.  For a complete classification of the characters of $A(\cC_v)$, see \cite{Mu}.

\subsection{Quantum trace characters}
First, we observe that for all $\rho\in\CC$, there exists a unique homomorphism of algebras:
$$\operatorname{tr}_\rho: T(A(\cC_v))\to\CC,$$
$$ l^i_j\mapsto \rho \delta^i_j.$$

It is easily checked that the left coideal subalgebra $U'\subset U$ is stable under $\operatorname{tr}_\rho$, in the following sense: for $x\in U'$, we have $x_{(1)}\operatorname{tr}_\rho(x_{(2)})\in U'$.  Thus for any $\xi:V\to\CC^\times$, we may define the character,
$$\operatorname{tr}_{\xi}:=\bigotimes_{v\in V} \operatorname{tr}_{\xi_v}:T(A(\cC))\to\CC.$$  We set $\mathcal{I}_\xi:=\ker \operatorname{tr}_{\xi} \subset U$.\glossary{$\mathcal{I}_\xi$}






\subsection{Multiplicative quantized quiver variety}
\begin{definition}
Fix a quiver $Q$, its dimension vector $\dd$, and character $\xi: V\to \CC^\times$.  The multiplicative, quantized quiver variety, $\mathcal{A}^\xi_\dd(Q)$, is the quantum Hamiltonian reduction of $\cD_q(\Mat_\dd(Q))$ by the moment map $\mu_q^\#$.  That is,\glossary{$\mathcal{A}^\xi_\dd(Q)$}
$$\mathcal{A}^\xi_\dd(Q) := \Hom_\cC\left(\mathbf{1}, \cD_q(\Mat_\dd(Q)) \Big/ \cD_q(\Mat_\dd(Q)) \mu_q^\#(\mathcal{I}_\xi)\right).$$
\end{definition}

\begin{definition} We let $\cD_q(\Mat_\dd(Q))$-mod$_\cC$ denote the category of $\cD_q$-modules in the category $\cC$.
\end{definition}

The following is a localization theorem for the algebras $\mathcal{A}_\dd^\xi(Q)$, whose proof is identical to that of \cite{GG2}, Corollary 7.2.4.  We refer the reader to the excellent exposition there.

\begin{theorem} \glossary{$\mathbb{H}$}We have an essentially surjective functor,
$$\mathbb{H}:\cD_q(\Mat_\dd(Q))\textrm{-mod}_\cC\to \mathcal{A}^\xi_\dd(Q)\textrm{-mod},$$
$$M\mapsto\Hom_\cC(\mathbf{1},M),$$
inducing an equivalence of categories,
$$\mathbb{H}:\cD_q(\Mat_\dd(Q))\textrm{-mod}_\cC / \mathrm{Ker}\,  \mathbb{H}\to \mathcal{A}^\xi_\dd(Q)\textrm{-mod}.$$
\end{theorem}

Here, $\mathrm{Ker}\, \mathbb{H}$ denotes the Serre subcategory of aspherical $\cD_q(\Mat_\dd(Q))$-modules, i.e. those modules whose space of invariants is zero.  The functor $\mathbb{H}$ is called the functor of Hamiltonian reduction.
\section{The degeneration of $A^\xi_\dd(Q)$}\label{degeneration}
\subsection{The Kassel-Turaev biquantization of $S(\g)$}
In order to compute the quasi-classical limit of $\cD_q(\Mat_\dd(Q))$ and its moment map $\mu_q^\#$, we need to recall from \cite{KT} the theory of biquantization of Lie bialgebras.  For $\g=\mathfrak{sl}_N$, the constructions we now recall here was also given by J. Donin \cite{Do}.  We begin with definitions.
\begin{definition}A co-Poisson algebra is a cocommutative coalgebra $C$, together with a Lie co-bracket $\delta:C\to C\wedge C$ satisfying the compatibility condition:
$$(\id\ot\Delta)\circ\delta = (\delta\ot \id + (\sigma\ot \id)\circ(\id\ot\delta))\circ\Delta.$$
\end{definition}
\begin{definition}
A bi-Poisson bialgebra is a commutative, cocommutative bialgebra $A$, together with a Poisson bracket and co-bracket, satisfying the compatibility conditions:
\begin{enumerate}
\item $\Delta(\{a,b\}) = \{\Delta(a),\Delta(b)\}$,
\item $\delta(ab) = \delta(a)\Delta(b) + \Delta(a)\delta(b)$,
\item $\delta(\{a,b\}) = \{\delta(a),\Delta(b)\} + \{\Delta(a),\delta(b)\}.$
\end{enumerate}
\end{definition}

Recall that for any vector space $V$, the symmetric algebra $S(V)$ is a bialgebra with coproduct:
$$\Delta(v)=v\ot 1 + 1\ot v.$$
A Lie bialgebra structure on $\g$ gives rise to a bi-Poisson bialgebra structure on the symmetric algebra $S(\g)$ by declaring the Poisson bracket and co-bracket be the unique extensions to $S(\g)$ of the Lie bracket and co-bracket on $\g$.  Consider $\g=\mathfrak{gl}_N$, and let
$$r:=\sum_{i<j} E^i_j\ot E^j_i + \frac12 \sum_i E^i_i\ot E^i_i\in \g\ot\g$$
denote the classical $r$-matrix for $\mathfrak{gl}_N$, associated to the trace form.  Of particular interest for us is the Lie bialgebra structure on $\mathfrak{gl}_N$, with cobracket $\delta: \g\to \g\ot \g$ given by:
$$\delta(x) := [r,x\ot 1 + 1\ot x],$$

In \cite{KT}, Kassel and Turaev constructed a $\CC[[u,v]]$ bialgebra $A(\g)=A_{u,v}(\g)$, which is a biquantization of $S(\g)$.  This means, firstly, that we have the following commutative diagram of bialgebras:
$$
\begin{CD} 
A(\g) @>>> A(\g)/(v)\\
@VVV @ VVV\\
A(\g)/(u) @>>> A(\g)/(u,v)\end{CD}$$
Secondly, we have natural isomorphisms  of coalgebras, algebras, and bialgebras, respectively:
$$A(\g)/(v)\cong S(\g)[[u]],\quad A(\g)/(u)\cong S(\g)[[v]],\quad A(\g)/(u,v)\cong S(\g).$$
In this sense, $A(\g)$ simultaneously quantizes the Poisson bracket and co-bracket on $S(\g)$: $v$ is the deformation parameter for the coproduct, and $u$ is the deformation parameter for the product.

Recall that the Etingof-Kazhdan quantization \cite{EK} of the Lie bialgebra $\g$ is a Hopf algebra $U_\hbar(\g)$, isomorphic as an algebra to $U(\g)$, but with coproduct which quantizes the co-bracket of $\g$.  Let  $V_u(\g):=A(\g)/(v)$, and let $A_{\hbar,\hbar}(\g)$ denote the quotient of $A_{u,v}(\g)$ by the ideal $(v-u)$ (in the quotient, we rename $\hbar:=u=v$   for notational convenience).  While we will not need to recall the full details of the construction of $A(\g)$, we will need the following descriptions of its quotients:
\begin{proposition}\cite{KT} \label{KTthm}
\begin{enumerate}
\item $A_{\hbar,\hbar}(\g)$ is the Etingof-Kazhdan quantization $U_{\hbar^2}(\g)$ of $\g$.\footnote{For $\g=\mathfrak{gl}_N$, this agrees with the Drinfeld-Jimbo quantization of $\mathfrak{gl}_N$.}
\item $V_u(\g) \cong T(\g) \Big / \langle X\ot Y - Y\ot X = u[X,Y] \,\, | \,\, X,Y\in \g\rangle.$
\end{enumerate}
\end{proposition}

Claim (1) is not explicitly stated in \cite{KT} but follows easily from the definition of $A_{u,v}(\g)$ given in Section 6, {\it loc. cit.}.  Claim (2) is Theorem 2.6.  Note that, by (2), we have an $\CC$-algebra homomorphism,
$$i:V_u(\g)\to U(\g)[[u]],$$
$$X\in \g\mapsto u X.$$

It follows by the PBW theorem that $i$ is an injection.  We may therefore identify $V_u(\g)$ with the Rees algebra of $U(\g)$, where the latter is filtered by declaring the generating subspace $\g$ to be degree 1.  

Let $\mathcal{U}_\hbar$ denote the $\CC[[\hbar]]$-Hopf algebra (a.k.a QUE algebra) obtained by setting $q=e^\hbar$ in Section \ref{sec:QG}.  We have the following well-known proposition:

\begin{proposition}
There exists an isomorphism $\alpha: \mathcal{U}_\hbar\to U[[\hbar]]$ of QUE algebras, such that $\alpha=\id \mod \hbar$.  Moreover, we have $\alpha(\mathcal{U}_\hbar')=V_\hbar(\g)$.
\end{proposition}

\begin{proof}
Recall that the generators $\tilde{l}^i_j$ of $\mathcal{U}_\hbar$ may be obtained as the matrix coefficients of the double-braiding:
$$(\id\ot\rho_{\CC^N})(R_{21}R) = \sum_{kl} \tilde{l}^k_l \ot E^l_k.$$
The claim now follows from the fact that $\alpha\ot\alpha(R_{21}R)\in V_\hbar(\g)^{\ot 2}$.
\end{proof}

\subsection{Flatness is preserved by quantum Hamiltonian reduction}\label{flatness-sec}
Throughout this section, we assume $Q$ and $\dd$ satisfy the conditions of Theorem \ref{CBflat}, so that the classical moment map $\mu:T^*\Mat_\dd(Q)\to\g^\dd$ is flat.  We set $q=e^\hbar$, and consider all algebras and categories defined in terms of $q$ to be defined over $\CC[[\hbar]]$, and complete in the $\hbar$-adic topology.  As a consequence of the flatness of $\mu$, we prove that the algebra $A_\dd^\xi(Q)$ is a flat formal deformation of its classical ($\hbar=0$) limit.  We note that similar results have been proven in \cite{Lo}, Lemma 3.6.1, and \cite{Br}.

To begin, we recall the following lemma from ring theory (see, e.g. \cite{B}, Chapter 2, Proposition 3.12):

\begin{lemma} Let $A_0$, be a graded ring, and $M_0$ a flat $A_0$-module.  Let $A$ be a ring with an exhaustive, increasing filtration, and $M$ an $A$-module with compatible filtration, such that $gr(A)\cong A_0$, and $gr(M)\cong M_0$ as $A_0$-modules.  Then $M$ is a flat $A$-module.
\end{lemma}
\begin{corollary}\label{assgrflat} Let $A_0, B_0$ be a graded rings, with a flat homomorphism $\phi_0:B_0\to A_0$ (i.e. $\phi$ makes $A_0$ into a flat left $B_0$-module).  Let $A,B$ be rings equipped with exhaustive, increasing filtrations, such that $gr(A)=A_0, gr(B)=B_0$. Then any filtered homomorphism $\phi:B\to A$ lifting $\phi_0$ is flat.
\end{corollary}

\begin{lemma}\label{A0lem} Let $A_0$ be a graded Poisson algebra with
a Poisson action of a reductive group $G$, and $\mu_0 :S\g\to A_0$
be a moment map for this action. Let A be a filtered algebra with
$gr(A)=A_0$, and $\mu: U(\g)\to A$ a quantum moment map that lifts
$\mu_0$ (so that the adjoint action is completely reducible).
If $\mu_0$ is flat, then so is $\mu$ (i.e. A is flat as a left $U(\g)$-module), and $gr(A//\g)=A_0//\g$.
\end{lemma}

\begin{proof}
The flatness of $A$ as a left $U(\g)$-module is an application of Lemma \ref{assgrflat}, with $B_0=S(\g)$ and $B=U(\g)$.  The Hamiltonian reduction $A//\g$ proceeds in two steps: first we construct the quotient $A/J$ of $A$ by its left ideal $J=A\mu(U(\g))\subset A$, and then we take the subspace of invariants in the quotient.  We show that each step is compatible with the filtration, and commutes with the associated graded construction.

The module $A/J$ inherits a filtration, and by flatness of $\mu$, we have $gr(A/J)=A_0/J_0$, where $J_0=A_0\mu_0(S(\g))$. Since the adjoint action of $\g$ on $A$ is completely reducible, and $J$ is $\g$ invariant, we have that the quotient $A/J$ embeds as a $\g$-submodule of $A$, and likewise $J^\g$ embeds as a submodule of $A^\g$.  Thus we have $(A/J)^\g\cong A^\g/J^\g$.  Finally, the action of $\g$ preserves the filtration on $A$, so we have:
$$gr(A/J)^\g \cong gr(A^\g/J^\g) \cong (A_0/ J_0)^\g = A_0//\g,$$
as desired.
\end{proof}
\begin{lemma}\label{defmod} Let $\mu_h$ be a deformation of the classical moment map $\mu$, $\mu_\hbar: U_\hbar(\g)\to A_\hbar$, where $A_\hbar$ is a flat deformation of $A$. Assume that the adjoint action is completely reducible. Then $\mu_\hbar$ is flat, and $A_\hbar//U_\hbar(\g)$ is a flat formal deformation (equivalently, it is torsion-free in $\hbar$).
\end{lemma}
\begin{proof}
First, we show that $\mu_\hbar$ is flat.  For this, we recall another lemma from ring theory.  While the proof is standard, we include it here for the sake of completeness.

\begin{lemma} Let $S$ be a (not necessarily commutative) flat formal deformation of the algebra $S_0=\CC[x_1,\ldots, x_n]$.  Let $\chi:S\to\CC[[\hbar]]$ be a character, specializing to $\chi_0:S_0\to \CC$.  Finally, suppose that $M$ is an $S$-module, topologically free over $\CC[[\hbar]]$, such that $M_0=M/\hbar M$ is flat over $S_0$.  Then $M\ot_S\chi$ is a flat formal deformation of $M_0\ot_{S_0}\chi_0$.
\end{lemma}

\begin{proof}
We denote by $\CC$ the one dimensional $\CC[[\hbar]]$-module, where $\hbar$ acts by zero.  We have only to check:
$$\operatorname{Tor}^i_{\CC[[\hbar]]}(M\ot_S\chi,\CC)\overset{?}{=}0.$$
Notice that we have an isomorphism, natural in $M$:
$$M\ot_S\chi\ot_{\CC[[\hbar]]}\CC\cong M\ot_{\CC[[\hbar]]}\CC\ot_{S_0}\chi_0 \cong M_0\ot_{S_0}\chi_0.$$
Thus, we have:
$$\operatorname{Tor}^i_{\CC[[\hbar]]}(M\ot_S\chi,\CC)\cong \operatorname{Tor}^i_{S_0}(M,\chi_0)=0,$$
by assumption of flatness on $M$.
\end{proof}

We now turn to proving the flatness of $A_\hbar//U_\hbar(\g)$.  We note that Hamiltonian reduction involves fixing a scalar action of $\mathfrak{gl}_N$, so that $A//G$ is completely reducible as a $U(\g)$-module.  By the flatness of $\mu_\hbar$, $A_\hbar/J_\hbar$ is a flat $\CC[[\hbar]]$-module.  Finally, complete reducibility gives an isomorphism $(A_\hbar/J_\hbar)^{U_\hbar(\g)}\cong (A/J)^\g[[\hbar]]$, as $\CC[[\hbar]]$-modules, because completely reducible $\g$-modules do not admit non-trivial deformations.   
\end{proof}

\begin{proposition}\label{Ilambdaflat} Let $t=\hbar$, and let $\xi_{v}:=e^{\hbar^2\lambda_{v}}$, for some $\xi:V\to\CC$.  Then quasi-classical limit of the ideal $\mathcal{I}_\xi$ is the classical moment ideal, i.e. the defining ideal of the closed set $\mu^{-1}(\sum \lambda_v\id_v)$.
\end{proposition}

\begin{proof}
The ideal $\mathcal{I}_\xi$ is generated by elements $\mu_q^\#(u)$, for $u\in \mathcal{U}_\hbar'$.  

Fix a $v\in V$, and let $r=|E_v|$.  We compute the image of $l^i_j\in \mathcal{U}_\hbar(\mathfrak{gl}^{d_v})$ under the map $\mu_v$ (see Section \ref{sec:QuivNot} for notation concerning quivers):
\begin{align*} \mu(l^i_j)&= \sum_{i_1,\ldots, i_r=1}^{d_v}\mu_v^{e_1}(l^i_{i_1})\mu_v^{e_2}(l^{i_1}_{i_2})\cdots \mu_v^{e_r}(l^{i_r}_j)\\
&= \delta^i_j + \hbar^2 \left(\sum_{e\in E^\beta_v}\sum_k \partial^i_ka^k_j - \sum_{e\in E^\alpha_v}\sum_k a^i_k\partial^k_j + \sum_{e\in E^{\circ}_v}\sum_k\left(\partial^i_ka^k_j - a^i_k\partial^k_j\right)\right) + O(\hbar^3).
\end{align*}
Thus the coefficient in $\hbar^2$ is precisely the LHS of equation \eqref{dppa}.  On the other hand, we easily compute that $\tr_\xi(l^i_j)=\delta^i_j + \hbar^2\lambda_v \delta^i_j$.  Thus equating $\hbar^2$ coefficients, we obtain Equation \eqref{dppa}.
\end{proof}

\begin{corollary} The algebra $A_\dd^\xi(Q)$ is a topologically free $\CC[[\hbar]]$-module, which is a flat formal deformation of $\cO(\mathcal{M}_\dd^\lambda$.
\end{corollary}

\begin{proof} First, we note that in the formal setting $\cD_q^\circ$ and $\cD_q$ coincide, as the $\operatorname{det}_q(e)$ are invertible formal power series.  We have shown in Theorem \ref{flatness} that $\cD_q$ is a flat formal deformation of $\cO(T^*\Mat_\dd(Q))$.  By applying Proposition \ref{Ilambdaflat}, we see that the ideal $\mathcal{I}_\xi$ deforms the classical moment ideal $I_\lambda$;  the deformation is flat by our assumptions on dimension vectors, and thus $A_\dd^\xi(Q)$ is a flat formal deformation of $\cO(\Mat_d(Q))\dS_{\hspace{-.05in}\xi} \mathbb{G}$ by Lemma \ref{defmod}.
\end{proof}
\section{Spherical DAHA's as quantized multiplicative quiver varieties}\label{sDAHA}

In this section we describe how to recover the spherical DAHA of type $A_{n-1}$ as the algebra $A_\dd^\lambda(Q)$, where $Q$ is the Calogero-Moser quiver, $(Q,d)= \overset{1}{\bullet}\rightarrow\overset{n}{\bullet}\rightloop.$  We also explain that the spherical generalized DAHA of type $Q$ is the algebra $A^\lambda_\dd(Q)$, when $Q$ is a star-shaped quiver.  As we have remarked in the Introduction, the results presented in this section, with formal parameters, are not very strong; in particular it would be interesting to upgrade the claims of this section to include generic numerical values of $q$, and also to study the parameter correspondence between the parameter $\lambda$ and the parameter $\mathbf{c}$ appearing in the definition of Cherednik algebras (see, e.g, \cite{EG}, \cite{EOR}).

\begin{lemma} (\cite{GG2}) The classical moment map,
$$\mu: \mathrm{Mat}_n\times\mathrm{Mat}_n\times \CC^n\times (\CC^n)^*\to \mathfrak{gl}_n(\CC)\times \CC,$$
$$(A,B,i,j)\mapsto ([A,B]+i\ot j,j(i)),$$
on the Calogero-Moser matrix space is flat.
\end{lemma}

We make use of the following lemma, which is proven in \cite{CEE}, using KZ functors, and in \cite{Ch1}, \cite{Ch2} by direct computation.

\begin{lemma}
The spherical DAHA of type $A_{n-1}$ is isomorphic as a $\CC[[\hbar]]$-algebra to the spherical Cherednik algebra of type $A_{n-1}$.
\end{lemma}

\begin{theorem} (\cite{EG}, Theorem 2.16)
The spherical Cherednik algebra is the universal deformation of the algebra of invariant differential operators on $\CC^n$ for the action of $S_n$.
\end{theorem}

\begin{theorem}\label{AnDAHA} The algebra $A^\xi_\dd(Q)$ is isomorphic to the spherical DAHA of type $A_{n-1}$.
\end{theorem}

\begin{proof}
Both algebras $A^\xi_\dd(Q)$ and the spherical DAHA of type $A_{n-1}$ are deformation quantizations of the Calogero-Moser variety.  Moreover, the spherical DAHA is the universal such deformation.  It follows that there exists a surjective homomorphism of $\CC[[\hbar]]$-algebras from spherical DAHA to $A^\xi_\dd(Q)$.  This map is the identity modulo $\hbar$, and is thus an isomorphism.
\end{proof}

\begin{theorem} Let $Q$ be a star-shaped quiver, and $\dd$ be the Calogero-Moser dimension vector of Example \ref{starexample}.  Then the algebra $A^\xi_\dd(Q)$ is isomorphic to the spherical GDAHA associated to $Q$.
\end{theorem}
\begin{proof}
This is proven in the same way as Theorem \ref{AnDAHA}.
\end{proof}
\end{document}